\documentclass[12pt,a4paper]{amsart}%

\usepackage{amsfonts,amsmath,amssymb}
\usepackage{amssymb,amsthm,amsxtra}
\usepackage[usenames]{color}
\usepackage{amscd}
\usepackage{amsthm}
\usepackage{amsfonts}
\usepackage{amssymb}
\usepackage{amsmath}
\usepackage{graphicx}
\usepackage{hyperref}
\usepackage{enumerate}%
\usepackage[all]{xy}
\usepackage[usenames,dvipsnames]{xcolor}
\usepackage{mathrsfs}
\usepackage{cleveref}
 \newtheorem*{corollary*}{Corollary}
 \newtheorem*{construction*}{Construction}
 \newtheorem*{definition*}{Definition}
 \newtheorem*{notation*}{Notation}
 \newtheorem*{lemma*}{Lemma}
 \newtheorem*{theorem*}{Theorem}
 \newtheorem*{remark*}{Remark}
 \newtheorem*{example*}{Example}
 \newtheorem*{conjecture*}{Conjecture}
 \newtheorem*{condition*}{Condition}
 \newtheorem*{result*}{Result}
 \newtheorem*{property*}{Property}

 \newtheorem*{cor*}{Corollary}
 \newtheorem*{const*}{Construction}
 \newtheorem*{defn*}{Definition}
 \newtheorem*{notn*}{Notation}
 \newtheorem*{lem*}{Lemma}
 \newtheorem*{thm*}{Theorem}
 \newtheorem*{rem*}{Remark}
 \newtheorem*{exm*}{Example}
 \newtheorem*{conj*}{Conjecture}

 \newtheorem{lemma}{Lemma}[section]
 
 \newtheorem{remark}[lemma]{Remark}
 
 \newtheorem{theorem}[lemma]{Theorem}
 
 \newtheorem{notation}[lemma]{Notation}

 \newtheorem{thm}[lemma]{Theorem}
 \newtheorem{prop}[lemma]{Proposition}
 \newtheorem{lem}[lemma]{Lemma}
 \newtheorem{defn}[lemma]{Definition}
 \newtheorem{notn}[lemma]{Notation}
 \newtheorem{cor}[lemma]{Corollary}

 \newtheorem{conj}[lemma]{Conjecture}
 \newtheorem{rem}[lemma]{Remark}

 \newtheorem{introtheorem}{Theorem}
 
 \crefname{introtheorem}{theorem}{theorems}
 \Crefname{introtheorem}{Theorem}{Theorems}
  \newtheorem{introthm}[introtheorem]{Theorem}
   \crefname{introthm}{theorem}{theorems}
 \Crefname{introthm}{Theorem}{Theorems}

  \crefname{introcorollary}{corollary}{corollaries}
 \Crefname{introcorollary}{Corollary}{Corollaries}

 \newtheorem{introcor}[introtheorem]{Corollary}
   \crefname{introcor}{corollary}{corollaries}
 \Crefname{introcor}{Corollary}{Corollaries}
 
   \crefname{introconjecture}{conjectures}{conjectures}
 \Crefname{introconjecture}{Conjecture}{Conjectures}
 
    \crefname{introconj}{conjectures}{conjectures}
 \Crefname{introconj}{Conjecture}{Conjectures}

     \crefname{introlem}{lemma}{lemmas}
 \Crefname{introlem}{Lemma}{Lemmas}
 
 \crefname{introremark}{remark}{remarks}
 \Crefname{introremark}{Remark}{Remarks}
 
  \crefname{introrem}{remark}{remarks}
 \Crefname{introrem}{Remark}{Remarks}

 \newtheorem{introprop}[introtheorem]{Proposition}
   \crefname{introprop}{Proposition}{Propositions}
 \Crefname{introprop}{Proposition}{Propositions}

   \crefname{introdefn}{definition}{definitions}
 \Crefname{introdefn}{Definition}{Definitions}
 
   \crefname{intronotn}{notation}{notations}
 \Crefname{intronotn}{Notation}{Notations}
 
   \crefname{introtask}{task}{tasks}
 \Crefname{introtask}{Task}{Tasks}
 
  \crefname{introprob}{problem}{problems}
 \Crefname{introprob}{Problem}{Problems}
 
   \crefname{introquestion}{question}{questions}
 \Crefname{introquestion}{Question}{Questions}

 \crefname{theorem}{theorem}{theorems}
 \Crefname{theorem}{Theorem}{Theorems}
  \crefname{thm}{theorem}{theorems}
 \Crefname{thm}{Theorem}{Theorems}

  \crefname{corollary}{Corollary}{Corollaries}
 \Crefname{corollary}{Corollary}{Corollaries}

   \crefname{cor}{Corollary}{Corollaries}
 \Crefname{cor}{Corollary}{Corollaries}

   \crefname{conjecture}{conjectures}{conjectures}
 \Crefname{conjecture}{Conjecture}{Conjectures}

    \crefname{conj}{conjectures}{conjectures}
 \Crefname{conj}{Conjecture}{Conjectures}

     \crefname{lem}{lemma}{lemmas}
 \Crefname{lem}{Lemma}{Lemmas}

      \crefname{lemma}{Lemma}{Lemmas}
 \Crefname{lemma}{Lemma}{Lemmas}

 \crefname{remark}{remark}{remarks}
 \Crefname{remark}{Remark}{Remarks}

  \crefname{rem}{remark}{remarks}
 \Crefname{rem}{Remark}{Remarks}

   \crefname{rem}{remark}{remarks}
 \Crefname{rem}{Remark}{Remarks}

   \crefname{proposition}{Proposition}{Proposition}
 \Crefname{proposition}{Proposition}{Proposition}

    \crefname{prop}{Proposition}{Propositions}
 \Crefname{prop}{Proposition}{Propositions}

   \crefname{defn}{definition}{definitions}
 \Crefname{defn}{Definition}{Definitions}

   \crefname{notn}{notation}{notations}
 \Crefname{notn}{Notation}{Notations}

   \crefname{task}{task}{tasks}
 \Crefname{task}{Task}{Tasks}

  \crefname{prob}{problem}{problems}
 \Crefname{prob}{Problem}{Problems}

   \crefname{question}{question}{questions}
 \Crefname{question}{Question}{Questions}

\newcommand{\lam}{\lambda}

\newcommand{\Oc}{\mathcal{O}}

\newcommand{\gr}{\operatorname{gr}}

\newcommand{\Ind}{\operatorname{Ind}}

\newcommand{\id}{\operatorname{Id}}

\renewcommand{\Im}{\operatorname{Im}}

\newcommand{\Span}{{\operatorname{Span}}}

\newcommand{\rk}{{\operatorname{rk}}}

\newcommand{\GL}{\operatorname{GL}}
\newcommand{\SL}{\operatorname{SL}}
\newcommand{\Sp}{\operatorname{Sp}}
\newcommand{\gl}{{\mathfrak{gl}}}

\newcommand{\Sym}{\operatorname{Sym}}

\newcommand{\ad}{\operatorname{ad}}

\newcommand{\WF}{\operatorname{WF}}
\newcommand{\Irr}{\operatorname{Irr}}

\renewcommand{\H}{\operatorname{H}}
\newcommand{\Tr}{\operatorname{Tr}}

\newcommand{\Hom}{\operatorname{Hom}}

\newcommand{\An}{\operatorname{An}}
\newcommand{\Ann}{\operatorname{Ann}}
\newcommand{\AnV}{\mathrm{An}\cV}
\newcommand{\Anv}{\AnV}
\newcommand{\As}{\operatorname{As}}
\newcommand{\Asv}{\As\cV}

\newcommand{\bC}{\mathbb{C}}
\newcommand{\C}{\mathbb{C}}

\newcommand{\bP}{\mathbb{P}}

\newcommand{\bZ}{\mathbb{Z}}

\newcommand{\bfG}{\mathbf{G}}
\newcommand{\bfH}{\mathbf{H}}

\newcommand{\Mat}{\operatorname{Mat}}

\newcommand{\R}{\mathbb{R}}

\newcommand{\ctp}{\widehat{\otimes}}
\newcommand{\Sc}{\cS}

\newcommand{\Fre}{Fr\'echet}

\newcommand{\FreSp}{Fr\'echet space}
\newcommand{\FreSpr}{Fr\'echet space }

\newcommand{\onto}{\twoheadrightarrow}
\newcommand{\into}{\hookrightarrow}

\providecommand{\fg}{\mathfrak{g}}
\providecommand{\fh}{\mathfrak{h}}

\providecommand{\fk}{\mathfrak{k}}
\providecommand{\fl}{\mathfrak{l}}

\providecommand{\fn}{\mathfrak{n}}

\providecommand{\fp}{\mathfrak{p}}

\providecommand{\fr}{\mathfrak{r}}
\providecommand{\fs}{\mathfrak{s}}

\providecommand{\fy}{\mathfrak{y}}

\providecommand{\cE}{\mathcal{E}}

\providecommand{\cM}{\mathcal{M}}
\providecommand{\cN}{\mathcal{N}}
\providecommand{\cO}{\mathcal{O}}

\providecommand{\cS}{\mathcal{S}}

\providecommand{\cU}{\mathcal{U}}
\providecommand{\cV}{\mathcal{V}}
\providecommand{\Kaz}{\mathrm{Kaz}}
\providecommand{\fin}{\sigma}

\providecommand{\g}{\mathfrak{g}}

\newcommand{\simpAr}[2][r]{%
\ar@{}[#1]|-*[@]_{#2}%
}

\newcommand{\Dima}[1]{{{#1}}}
\newcommand{\DimaA}[1]{{{#1}}}
\newcommand{\DimaB}[1]{{{#1}}}
\newcommand{\DimaC}[1]{{{#1}}}
\newcommand{\DimaD}[1]{{{#1}}}
\newcommand{\DimaE}[1]{{{#1}}}
\newcommand{\DimaF}[1]{{{#1}}}
\newcommand{\nextversion}[1]{} 

\newcommand{\proofend}{\hfill$\Box$\smallskip}

\newcommand{\hot}{\,\widehat{\otimes}\,}

\begin{document}

\title{Annihilator varieties of distinguished modules of reductive Lie algebras}
\author{Dmitry Gourevitch}
\address{Dmitry Gourevitch, Faculty of Mathematics and Computer Science, Weizmann
Institute of Science, POB 26, Rehovot 76100, Israel }
\email{dimagur@weizmann.ac.il}
\urladdr{http://www.wisdom.weizmann.ac.il/~dimagur}

\author{Eitan Sayag}
\address{Eitan Sayag,
 Department of Mathematics,
Ben Gurion University of the Negev,
P.O.B. 653,
Be'er Sheva 84105,
ISRAEL}
 \email{eitan.sayag@gmail.com}
\address{Ido Karshon, Faculty of Mathematics and Computer Science, Weizmann
Institute of Science, POB 26, Rehovot 76100, Israel }
\email{ido.karshon@gmail.com}

%
%
%
%
%
%
%
%
\date{\today}

\maketitle
\begin{abstract}
We provide a micro-local necessary condition for distinction of admissible representations of real reductive groups in the context of spherical pairs.

Let $\bfG$ be a complex algebraic reductive group, and $\bfH\subset \bfG$ be a spherical algebraic subgroup. Let $\fg,\fh$ denote the Lie algebras of $\bfG$ and $\bfH$, and let $\fh^{\bot}$ denote the orthogonal complement to $\fh$ in $\fg^*$.
A $\fg$-module is called $\fh$-distinguished if it admits a non-zero $\fh$-invariant functional. We show that the maximal $\bfG$-orbit in the annihilator variety of any irreducible $\fh$-distinguished $\fg$-module intersects $\fh^{\bot}$. This generalizes a result of Vogan \cite{Vog}.

We apply this to Casselman-Wallach representations of real reductive groups to obtain information on branching problems, translation functors and Jacquet modules. Further, we prove in \Dima{many} cases that  as suggested by \cite[Question 1]{PraSak}, when $H$ is a symmetric subgroup of a real reductive group $G$, the existence of a tempered $H$-distinguished representation of $G$ implies the existence of a generic $H$-distinguished representation of $G$.

Many of the models studied in the theory of automorphic forms involve an additive
character on the unipotent radical of the subgroup $\bf H$, and we devised a twisted version of our theorem that yields necessary conditions for the existence of those mixed models.
Our method of proof here is \DimaA{inspired by} the theory of modules over $W$-algebras. As an application of  our theorem we derive necessary conditions for the existence of Rankin-Selberg, Bessel, Klyachko and Shalika models.
Our results are compatible with the recent Gan-Gross-Prasad conjectures for non-generic representations \cite{GGP}.

\DimaE{Finally, we provide more general results that ease the sphericity assumption on the subgroups, and apply them to local theta correspondence in type II and to degenerate Whittaker models.}

%
%
%
%
%
\end{abstract}
%
%


\section{Introduction}\label{sec:intro}


In recent years, the study of periods of automorphic forms and of distinguished representations received a lot of attention. We mention here the work of Jacquet connecting distinguished representations via the {\it relative trace formula} to the image of Langlands functoriality \cite{Jac}, and the conjectures of Gan, Gross and Prasad \cite{GGPW} describing branching laws (for classical groups) in terms of \DimaF{Arthur} parameters.

More recently, the work of Sakellaridis and Venkatesh \cite{SV} regarding harmonic analysis on spherical varieties lead to a conjectural parameterization of classes of distinguished representations and \cite{Wan} pushed the conjectures of Gan, Gross and Prasad beyond classical groups to a  more general set up of certain spherical pairs.
In these works, the question whether a representation has an invariant functional is tied with the Langlands program.

In the present paper we establish a simple geometric necessary condition for an irreducible representation $\pi$ of a reductive group $G$ to admit a non-zero invariant functional with respect to a spherical subgroup $H.$

Our methods are Lie theoretic and based on the relationship between representations and nilpotent orbits in the spirit of the orbit method.
This relationship is familiar in the study of $U(\fg)$ modules and extends to modules over $W$-algebras, allowing us to consider certain twisted models.
We believe this brings a new aspect to the study of distinguished representations.

To formulate our result we require some notation.
Let $\bfG$ be a connected complex algebraic reductive group, and $\bfH\subset \bfG$ be a spherical algebraic subgroup. Let $\fg,\fh$ denote the Lie algebras of $\bfG$ and $\bfH$.
Denote by $\Irr(\fg)_\fh$ the collection of simple $\fg$-modules $V$ that have a non-zero $\fh$-invariant functional. We will call these modules \emph{$\fh$-distinguished}. Equivalently, these are modules with $\H_0(\fh,V)\neq 0$. Denote by $\cM_{f.d.}(\fh)$ the category of finite-dimensional $\fh$-modules.
For any $\fin\in \cM_{f.d.}(\fh)$, we denote by $\Irr(\fg)_{(\fh,\fin)}$ the collection of simple $\fg$-modules $V$ with $\H_{0}(\fh,V\otimes \fin)\neq 0$.

For any $\fg$-module $V$, denote by $\An\cV(V)$ the associated variety of the annihilator of $V$, that we will call shortly \emph{annihilator variety}. To  shorten our formulations we will use the following theorem.

\begin{thm}[{\cite{BB3,Jos}, cf. \cite[Corollary 4.7]{Vog}}]\label{thm:Irr}
For any $V\in \Irr(\fg)$, there exists a unique coadjoint nilpotent orbit $\cO$ such that $\Anv(V)=\overline{\cO}$.
\end{thm}
We will then denote this orbit by $\cO(V)$.
Denote by $\fh^{\bot}$ the orthogonal complement to $\fh$ in $\fg^*$, and by $\cO(V)\cap\fh^\bot$ the  intersection of $\fh^{\bot}$ with $\cO(V)$.

\begin{introthm}[\S \ref{sec:Pfmain}]\label{thm:main}
For all $\fin\in \cM_{f.d.}(\fh)$ and all
$V\in \Irr(\fg)_{(\fh,\fin)},$
 we have
 \begin{equation*}
2 \dim (\cO(V)\cap \fh^\bot)=\dim \cO(V)
 \end{equation*}
 In particular, the intersection is not empty.
\end{introthm}


\DimaA{Let us now present a twisted version of Theorem \ref{thm:main}.}
%
%

\begin{introthm}[\S\ref{sec:twist}]\label{thm:maintwist}
Let $\fp$ be a parabolic subalgebra of $\fg$ with Levi decomposition $\fp=\fl\oplus \fn$. Let $\chi$ be \DimaA{a} character of $\fn$.
Let $\fs$   be a spherical Lie subalgebra of $\fl$, and let $\fh:=\fs\oplus \fn$. Let $\fin\in \cM_{f.d.}(\fh)$ such that $\fn$ acts on $\fin$ via $\chi$.
Extend $\chi$ to an element of $\fg^*$ that vanishes on $\fs$.

Then for any $V\in \Irr(\fg)_{(\fh,\fin)}$ we have
$$2\dim \left( \cO(V)\cap (\chi+\fh^{\bot})\right)=\dim \cO(V).$$ In particular, the intersection is non-empty.
\end{introthm}
\DimaA{ If $\chi$ is a generic character of $\fn$,} we will say that the triple $(\fg,\fh,\sigma)$ is the \emph{Whittaker induction} of the triple $(\fl,\fs,\sigma|_{\fs})$. This notion is different from a similar notion  in \cite[\S 2.6]{Wan}.

Let us now emphasize an application of Theorems \ref{thm:main} and \ref{thm:maintwist} to real reductive groups. Assume that $\bfG$ is defined over $\R$, and let $G$ be a finite cover of an open subgroup of the group of real points of $\bfG$.
In \S \ref{sec:CW} below we recall the notion of a Casselman-Wallach representation of $G$, as well as some basic properties of such representations.
Denote the collection of irreducible Casselman-Wallach representations of $G$ by $\Irr(G)$.
By Theorem \ref{thm:Irr} and the Casselman-Wallach equivalence of categories (see \S \ref{sec:CW}), $\An\cV(\pi)$ is the closure of a unique coadjoint nilpotent orbit $\cO(\pi)$ for any $\pi\in \Irr(G)$.
We deduce from Theorems \ref{thm:main} and \ref{thm:maintwist} the following corollary.

\begin{introcor}[\S \ref{sec:CW}]\label{cor:CW}
Let $\fin\in \cM_{f.d.}(\fh)$  and  let $\pi\in \Irr G$. Assume that $\pi$ has a non-zero continuous $\fh$-equivariant map into $\fin$. Then  $2\dim\cO(\pi)\cap \fh^\bot=\dim\cO(\pi)$.

Furthermore, assume that there exist a parabolic subalgebra $\fp$ of $\fg$ with Levi decomposition $\fp=\fl\oplus \fn$, and
$\chi\in \fg^*$ such that $\fh=(\fh\cap \fl)\oplus\fn$,  $\chi|_{\fl}=0$, 
\DimaA{and
$\fn$ acts on $\fin$ via} $\chi|_\fn$. Then  $$2\dim \left(\cO(\pi)\cap (\chi+\fh^{\bot})\right)=\dim \cO(\pi).$$
\end{introcor}

The first natural generality to apply this corollary is symmetric pairs, since for them the nilpotent orbits intersecting $\fh^{\bot}$ are classified in \cite{Djo1,Djo2,Oht}. In \S \ref{subsec:Pra} below we deduce a partial affirmative answer to \cite[Question 1]{PraSak}. Namely, we show  that for \Dima{many} Archimedean symmetric pairs, the existence of an $H$-distinguished tempered irreducible representation implies the existence of an $H$-distinguished generic irreducible representation.

In \S \ref{subsec:bra} below we  apply Corollary \ref{cor:CW} to branching problems, and to translation functors. Then, in \S \ref{subsec:Jac}, we  deduce from Corollary \ref{cor:CW} restrictions on annihilator varieties of Jacquet modules.

\DimaA{
Let us now present a twisted version of the results on branching problems and Jacquet functors. Let $P=LN$ be a Levi decomposition of a parabolic subgroup of $G$.
Let $\chi$ be a character of $\fn$. Suppose that there exists a reductive subgroup $S\subset L$ that stabilizes $\chi$ such that $\Delta S \subset S\times L$ is a spherical subgroup.
Let $\fr:=\fs\oplus \fn \subset \fp$.
Let $\tau\in \Irr(S)$, and consider it as an $\fr$-module via the  projection $\fr\onto \fs$. Extend $\chi$ to a character of $\fr$ vanishing on $\fs$. For a subset $X\subset \fg^*$ let $X|_{\fr}\subset \fr^*$ denote the restriction of $X$ to $\fr$. Consider $\fs^*$ as a subset of $\fr^*$.

\begin{introcor}\label{cor:TJ}
Let $\pi\in \Irr(G)$ and $\tau \in \Irr(S)$ be such that $\Hom_{\fr}(\pi,\tau\otimes \chi)\neq 0$.

Then
$\cO(\Dima{\tau})$ intersects $\cO(\pi)|_{\fr}+\chi$.
\end{introcor}

One special case of this corollary is $G=\GL_{n+k}(\R), \, L=\GL_{n+1}(\R)\times (\R^{\times})^{k-1}, \, S=\GL_{n}(\R)$, and the character $\chi$ is generic. We call the models arising from this case \emph{Rankin-Selberg} models. Similar models for orthogonal or unitary groups are called \emph{Bessel models}. In \S\S \ref{subsec:RS},\ref{subsec:Bessel} we deduce from Corollary \ref{cor:TJ} a microlocal necessary condition on the existence of those models. We now give a brief uniform formulation, and refer to \S\S \ref{subsec:RS},\ref{subsec:Bessel} for further details.
}


\begin{introthm}[\S\S \ref{subsec:RS},\ref{subsec:Bessel}]\label{thm:strongPart}
Let $S$ be from the list
\begin{equation}\label{=Strong}
\GL_n(\R), \quad \GL_n(\C), \quad \mathrm{U}(m,n)\quad \mathrm{O}(m,n),\quad \mathrm{SO}(m,n),\quad  \mathrm{SO}_{m+n}(\C), \quad  \mathrm{O}_{m+n}(\C)
\end{equation}
and $L'$ be the corresponding group with $n$ replaced by $n+1$.
Embed $S$ in the top left corner of $L'$.
Let $G$ be the group from \eqref{=Strong} corresponding to $S$,  with
$n$ replaced by $n+k$ and $m$ replaced by $m+k\Dima{-1}$, for some $k\geq \Dima{1}$.

Let $P=LN$ be the standard parabolic subgroup of $G$ with $L=L'\times \Dima{(\R^{\times})^{k-1}}$ or $L=L'\times \Dima{(\C^{\times})^{k-1}}$.
Let $\phi$ be a generic unitary character of $N$ stabilized by $S$.

Let  $\pi\in \Irr(G)$ and $\tau\in \Irr(S)$, and let $\lambda$ and $\mu$ be the partitions corresponding to $\cO(\pi)$ and $\cO(\tau)$ respectively.
Let $\lambda^t$ and $\mu^t$ denote the transposed partitions.

Suppose that $\Hom_{SN}(\pi|_{SN},\tau\otimes \phi)\neq 0$.
Then
for any index $i\geq 1$ we have
$|\lambda^t_i-\mu^t_i|\leq 1$.
\end{introthm}
\Dima{For example, the case $k=1$ corresponds to branching problems for the restrictions from $\GL_{n+1}(\R)$ to $\GL_{n}(\R)$, restrictions from $\mathrm{SO}(m,n+1)$ and $\mathrm{SO}(m,n)$, and analogous restrictions for complex groups and for unitary groups. }

This theorem partially confirms the non-tempered Gan-Gross-Prasad conjectures \cite{GGP} in the Archimedean case. For the case of $\GL_{n}(F)\subset \GL_{n+k}(F)$ for any $p$-adic field $F$, the conjectures were established in \cite{Chan,MGur}. In the subcategory of unitary representations (which is not a full subcategory), more restrictive conditions are proven in \cite{Ven,Hen}.
\DimaB{In the language of \cite[\S 22.2]{NV}, Corollaries \ref{cor:CW} and \ref{cor:TJ} state that distinction implies orbit-distinction near infinity in the dual Lie algebra.}

\Dima{ In} \cite{JSZ,CS} the space $\Hom_{SN}(\pi|_{SN},\tau\otimes \phi)$ is shown to \Dima{have dimension at most one}.

Over finite fields, the branching problem from $\GL_{n+1}(\mathbf{F}_q)$ to $\GL_{n+1}(\mathbf{F}_q)$ is fully analysed in \cite{Tho}, in terms of the partition-valued functions on the set of so-called simplexes, that classify the irreducible representations by \cite[Theorem 13]{Green}. In particular, \cite[Satz 2]{Tho} implies
that for representations $\pi$ and $\tau$ given by functions $f$ and $h$, $\Hom(\pi|_H,\tau)\neq 0$ if and only if for every simplex $s$ and every index $i$ we have $|f(s)_i-h(s)_i|\leq 1$.

In \cite{uni}, the authors consider the analogous branching problems for unipotent representations of unitary groups over finite fields. These representations are described in terms of partitions, and the authors call partitions $\lambda$ and $\mu$ \emph{close} if they satisfy $|\lambda_i-\mu_i|\leq 1$ for every $i$. They show that $\Hom(\pi_{\lambda}|_H,\pi_{\mu})\neq 0$
if and only if $\lambda$ and $\mu$ are close and the multi-set of their common parts has even multiplicities.

In \S\S \ref{subsec:Klya},\ref{subsec:Sha}, below we also apply Corollary
 \ref{cor:CW} to obtain necessary conditions for the existence of Klyachko models and Shalika models for irreducible representations of $\GL_n(\R)$ and $\GL_n(\C)$. By \cite{GOSS},  in the case of unitarizable representations our necessary condition for the existence of Klyachko models is also sufficient.

\DimaC{We also} provide a non-homogeneous analogue of Corollary \ref{cor:CW}.
Let $\bf X$ be a spherical smooth $\bfG$-variety defined over $\R$.
Let us recall the definition of the moment map $\mu_{\bf X}: T^*\bf X \to \fg^*$. For every $x\in {\bf X},$ the differential action map $\bf G\to X$ is a linear map $\fg\to T_x{\bf X}$. Dualizing it, and running over all points of $\bf X$ we obtain the moment map. Let $\Im(\mu_{\bf X})$ denote the image of the moment map. Note that if $\bf X=G/H$ then $\Im(\mu_{\bf X})=\bfG\cdot \fh^{\bot}$.

Let $X$ be a union of connected components of the manifold $\bf X(\R)$. Let $\cE$ be an algebraic bundle over $X$, and let $\Sc(X,\cE)$ denote the space of Schwartz sections of $\cE$ (see \S \ref{subsec:Sc} below for the definition).
\begin{introthm}[\S \ref{subsec:PfAltX}]\label{thm:AltX}
Let $\pi\in \Irr(G)$. If  $\H_0(\fg,\Sc(X,\cE)\hot \pi)\neq 0$ then $\cO(\pi)\subset \Im(\mu_{\bf X})$.
\end{introthm}

\DimaC{Finally let us formulate a proposition that holds for any Lie subalgebra $\fr\subset \fg$, possibly non-spherical.

\begin{introprop}[\S \ref{sec:Pfmain}]\label{prop:GenMain}
Let $\fr\subset \fg$ be any subalgebra. Then for all $\fin\in \cM_{f.d.}(\fr)$ and all
$V\in \Irr(\fg)_{(\fr,\fin)},$
 we have
 \begin{equation}
2 \dim \AnV(V)\cap \fr^{\bot}\geq \dim \AnV (V)
 \end{equation}
\end{introprop}

This proposition allows to replace in Theorems \ref{thm:main} and \ref{thm:maintwist}, and thus also in Corollary \ref{cor:CW} the assumption that $\fh$ is spherical by the assumption that $2\dim \cO\cap \fh^{\bot}\leq \dim \cO$ for any $\cO\subset \An \cV(V)$. We formalize this in \S \ref{sec:beyond} below, and apply to theta correspondence in type II \DimaE{and to degenerate Whittaker models}. }

\subsection{Conjectures}

We  conjecture that the non-Archimedean analogue of Corollary \ref{cor:CW} holds true. In the  non-Archimedean case, instead of the annihilator variety one uses the Zariski closure of the wave-front set of the (distribution) character of $\pi$ (see Notation \ref{notn:WF} below).
In the Archimedean case, the Zariski closure of the wave-front set of any Casselman-Wallach representation coincides with the annihilator variety by \cite[Theorem D]{Ross}.

However, the non-Archimedean analogue of Theorem \ref{thm:Irr}  is not known in general, though it is conjectured. Thus, we have two options for a  non-Archimedean analogue of Corollary \ref{cor:CW}: to state that all top orbits in the Zariski closure of the wave-front set intersect $\fh^{\bot}$ (or $\fh^{\bot}+\chi$ in the twisted case), or that some top orbit intersects.
If we opt for the stronger formulation, and apply it to the diagonal pair $\Delta G \subset G\times G$, we obtain  the statement that the top stable orbit is unique.

A certain partial $p$-adic  analogue of Theorem \ref{thm:main} is proven in \cite{GS_X}. We discuss it, as well as the conjectures, their corollaries and some partial evidence in \S \ref{sec:conj}.

\subsection{Structure of the paper}

In \S \ref{sec:Prel} we give the necessary preliminaries on associated varieties, annihilator varieties, and spherical subgroups. In particular, the theorem of Gabber and Joseph stating that the dimension of the associated variety is at least half the dimension of the annihilator variety, and a theorem of Li stating that for any nilpotent orbit $\cO\subset \fg^*$, the intersection $\cO\cap \fh^{\bot}$ is isotropic in $\cO$, and thus has dimension at most $\dim\cO/2$.

In \S \ref{sec:Pfmain} we prove Theorem \ref{thm:main} in the following way.
The non-vanishing of $\H_0(\fh,V\otimes \sigma)$ implies the existence of a non-zero $\fh$-invariant map $T:V\to \sigma^*$.
 Denote by  $M\subset \Hom_{\C}(V,\sigma)$ the submodule generated by $T$. Since $V$ is irreducible, $M$ has a non-degenerate $\sigma^*$-valued pairing with $V$. Thus $M$ and $V$ have the same annihilator variety. By Theorem \ref{thm:Irr} this variety is the closure of some nilpotent orbit $\cO$. Since $M$ is generated by an  $\fh$-finite vector, its associated variety lies in $\overline{\cO}\cap \fh^{\bot}$. The theorems of Gabber-Joseph and of Li now give inequalities $$\dim\overline{\cO}\cap \fh^{\bot}\leq \dim \cO/2\leq \dim\overline{\cO}\cap \fh^{\bot},$$ implying $\dim\overline{\cO}\cap \fh^{\bot}=\dim \cO/2.$ Applying Li's theorem to orbits $\cO'\subset \overline{\cO}$, we obtain that $\cO$ intersects $\fh^{\bot}$, and the intersection has dimension $\dim\cO/2$.

In \S \ref{sec:twist} we prove a generalization of Theorem \ref{thm:maintwist}, using \DimaA{the Kazhdan filtration from the theory of $W$-algebras, and Gabber's theorem that says that the associated variety is a coisotropic subvariety of the annihilator variety.}

In \S \ref{sec:CW} we give the necessary preliminaries on Casselman-Wallach representations, and deduce Corollary \ref{cor:CW} from Theorems \ref{thm:main} and \ref{thm:maintwist} by applying them to Harish-Chandra modules.

In \S \ref{sec:appl} we give first applications of Corollary \ref{cor:CW}.
First, we apply Corollary \ref{cor:CW} to symmetric pairs and \Dima{partially} answer \cite[Question 1]{PraSak} for Archimedean pairs, using also \cite{Djo1,Djo2,Oht,PT,Harris,GS_X}. Then, we apply Corollary \ref{cor:CW} to branching problems.
Then we give restrictions on annihilator varieties of irreducible quotients of Jacquet modules. Finally, we treat irreducible quotients of twisted Jacquet modules, proving Corollary \ref{cor:TJ}.

In \S \ref{sec:TwistAppl} we apply Corollaries \ref{cor:TJ} and \ref{cor:CW} to give necessary conditions on existence of several mixed models widely used in the theory of automorphic forms: Rankin-Selberg, Bessel, Shalika,  Klyachko, \Dima{and Ginzburg-Rallis} models.

In \S \ref{sec:X} we deduce Theorem \ref{thm:AltX} from Corollary \ref{cor:CW} in the following way. The Casselman embedding theorem and \cite[Theorem C]{AGKL} imply that $\H_0(\fg,\Sc(X,\cE)\hot \pi)$ is Hausdorff and finite-dimensional. Then, the case of transitive action of $\bf G$ on $\bf X$ follows from Corollary \ref{cor:CW} by a version of Frobenius reciprocity for small induction. The general case follows by induction on the number of $\bf G$-orbits on $\bf X$ using the theory of Schwartz functions and tempered distributions.

\DimaC{
In \S \ref{sec:beyond} we formulate generalizations of the main results that replace the sphericity assumption by a dimension assumption. We apply these generalizations to theta correspondence in type II, using a geometric lemma which is postponed to Appendix \ref{app:Ido} by Ido Karshon. \DimaE{As an additional application, we give a new proof of the result of \cite{Mat} that provides a necessary condition for the existence of Whittaker models. }
}

In \S \ref{sec:conj} we discuss proven and conjectural analogues of Corollary \ref{cor:CW} and their applications.

\subsection{Acknowledgements}

We thank Joseph Bernstein, Shachar Carmeli, Maxim Gurevich,  Erez Lapid, \Dima{Ivan Losev, Paul Nelson, Dipendra Prasad, Alexander Shamov,} and Aviv Taller for fruitful discussions.  We also thank Alisa Qui  for proving Lemma \ref{lem:quot}\DimaC{, Ido Karshon for writing Appendix \ref{app:Ido}, and the anonymous referees for useful remarks}.

 D. G. was partially supported by ERC StG 637912, BSF grant 2019724, and  ISF grant 249/17.

\section{Preliminaries on associated and annihilator varieties}\label{sec:Prel}
Let $\cU(\fg)$ denote the universal enveloping algebra of $\fg$. While $\cU(  \mathfrak{g})  $ is
not commutative, it admits a natural Poincare-Birkhoff-Witt filtration $\cU_i(\fg)$, such that the associated graded algebra $\gr\left(\cU
\mathfrak{g})  \right)  $ is the symmetric algebra $S(
\mathfrak{g})  ,$ and one has a ``symbol" map $\sigma$ from $\cU(
\mathfrak{g})  $ to $S\left(  \mathfrak{g}\right)  $.
Note that $S(\fg)$ is the algebra of polynomials on $\fg^*$.

\begin{defn}
For a $\fg$-module $V$ let its \emph{annihilator} $\Ann(V)\subset \cU(\fg)$ denote the two-sided ideal consisting of elements that act by zero on $V$.

The \emph{annihilator variety} of $V$ is the set of common zeros in $\fg^*$ of the set of symbols of all the elements of $\Ann(V)$:
\[
\AnV(V)  =\text{Zeroes}\left(  \sigma (\left(  \Ann(V)\right)
\right)  \subset\mathfrak{g}^{\ast}%
\]
\end{defn}

\begin{notn}
Denote by $\cN(\fg^*)$ the set of nilpotent elements of $\fg^*$.
\end{notn}

By \cite[Theorem 9.1]{KosNil}, $\cN(\fg^*)$ is the set of common zeros of all the homogeneous $ad(\fg)$-invariant elements of $S(\fg)$ of positive degrees. Such elements are precisely the symbols of non-constant elements of the center of $\cU(\fg)$. Thus the annihilator variety of any $V\in \Irr(\fg)$ lies in  $\cN(\fg^*)$.

Denote by $\cM_{f}(\fg)$ the category of finitely-generated $\fg$-modules.

\begin{defn}

For any $V\in \cM_{f}(\fg)$, define its \emph{associated variety} $\Asv(V)$ in the following way. Let $S:=\{m_1,...,m_k\}$ be a set of generators of $V$, and define a filtration $V_i$ of $V$ by $V_i:=\cU_i(\fg)S$. Let $\gr V$ be the associated graded $S(\fg)$-module, and let $\Ann(\gr V)$ be the annihilator ideal of $\gr V$ in $S(\fg)$. Then $\Asv(V)$ is the set of zeros of $\Ann(\gr V)$ in $\fg^*$.
\end{defn}
It follows from the definitions that $\Asv(V)\subset \Anv(V)$.
For more information on $\Asv(V)$, and in particular the proof that it does not depend on the choice of generators, we refer the reader to \cite{Vog}.

\begin{thm}[Gabber-Joseph, see {\cite[Theorem
9.11]{KrLe}}]\label{thm:ineq}
For any $V\in \cM_{f}(\fg)$, $$2\dim \Asv(V)\geq \dim \Anv(V).$$
\end{thm}

\begin{lem}\label{lem:AsH}
If $V\in \cM(\fg)$ is generated by a finite-dimensional $\fh$-invariant subspace $V_0$ then $\Asv(V)\subset \fh^{\bot}$.
\end{lem}
\begin{proof}
Define a filtration on $V$ by $V_i:=\cU_i(\fg)V_0$. Then we have $$\fh V_i\subset [\fh,\cU_i(\fg)]V_0+\cU_i(\fg)\fh V_0\subset V_i.$$
Thus, the symbols of $\fh$ act by zero on $\gr V$, and thus $\Asv(V)\subset \fh^{\bot}$.
\end{proof}
\Dima{The converse also holds (if $V$ is finitely-generated). Another equivalent condition is that $\fh$ acts on $V$ locally finitely.}

\begin{defn}
An algebraic subgroup $\bfH\subset \bfG$ is called \emph{spherical} if its action on the flag variety of $\bfG$ has an open orbit.
\end{defn}
The spherical subgroups have been extensively studied and eventually classified, see {\it e.g.} \cite[Chapter 5]{Tim} for an exposition. In particular, a classical theorem of Brion and Vinberg says that any spherical subgroup has  finitely many orbits on the flag variety. However, the only fact on spherical subgroups that we will use is the following theorem.

\begin{thm}[{\cite[Theorem 3.8]{Li}}]\label{thm:Geo}
For any nilpotent coadjoint orbit $\cO\subset \cN(\fg^*)$ and any spherical subgroup $\bfH \subset \bfG$, the intersection $\cO\cap \fh^{\bot}$ is an isotropic subvariety of $\cO$ with respect to the
Kirillov-Kostant-Souriau symplectic structure.
\end{thm}
For the definition of the Kirillov-Kostant-Souriau symplectic form see e.g. \cite[Proposition 1.1.5]{CG}. \Dima{We will only need the corollary that if $\cO\cap \fh^{\bot}$ is non-empty then its dimension is at most half the dimension of $\cO$. }
%
%

\begin{rem}\label{rem:Lag}
\begin{enumerate}[(i)]
\item The intersection $\cO\cap \fh^{\bot}$ is frequently empty.
\item It is possible that $\cO\cap \fh^{\bot}$ is always Lagrangian in $\cO$. This is the case when $\bfH$ is a solvable (spherical) subgroup of $\bfG$ (see \cite[Theorem 1.5.7]{CG}), and when $\bfH$ is a symmetric subgroup of $\bfG$ (see \cite[Proposition 3.1.1]{Gin}).
\item \label{it:LagEq} \Dima{By Theorem \ref{thm:Geo} and \cite[Corollary 5.18]{Vog}} the
 following are equivalent:
 \begin{enumerate}
 \item $\cO\cap \fh^{\bot}$ includes a non-empty Lagrangian subvariety of $\cO$.
 \item $\dim \cO=2\dim \cO\cap \fh^{\bot}$.
 \item $\bf H$ acts on $\cO\cap \fh^{\bot}$ with an open orbit.
 \item The scheme-theoretic intersection of $\cO$ with $\fh^{\bot}$ \Dima{has a reduced non-empty open subscheme}.
 \end{enumerate}
 \end{enumerate}
\end{rem}

\section{Proof of Theorem \ref{thm:main} and \Cref{prop:GenMain}}\label{sec:Pfmain}

\DimaC{

\begin{proof}[Proof of \Cref{prop:GenMain}]
Since $V\in \Irr(\fg)_{\fr,\sigma}$,  there exists a non-zero $\fr$-invariant map $\xi:V\to \fin^*$. Consider $\Hom_{\C}(V,\fin^*)$ as a $\fg$-module via the action on the argument, and let $M\subset \Hom_{\C}(V,\fin^*)$ be the $\fg$-submodule generated by $\xi$. Then we have a natural  $\fin^*$-valued $\fg$-invariant bilinear form $V\times M\to \fin^*$. This form has no right kernel by the definition of $M$, and
no left kernel since $V$ is a simple module. Thus $M$ and $V$ have the same annihilator ideal in $\cU(\fg)$ and thus
$$\Anv(M)=\Anv(V)=\overline{\cO}.$$

Since $\xi$ is $\fr$-equivariant and $\fin^*$ is finite-dimensional, we get that $\xi$ generates a finite-dimensional $\fr$-submodule of $M$ under the action of $\fr$ on the argument. By Lemma \ref{lem:AsH}, this implies $\Asv(M)\subset \fr^{\bot}$. Thus
\begin{equation}\label{=Assh}
\Asv(M)\subset \Anv(M)\cap \fr^{\bot}
\end{equation}

By Theorem \ref{thm:ineq} we have
\begin{equation}\label{=dimAss}
2\dim \Asv(M)\geq \dim \Anv(M).
\end{equation}
Altogether we get
\begin{equation*}
2 \dim \AnV(V)\cap \fr^{\bot}=2 \dim \AnV(M)\cap \fr^{\bot}\geq 2\dim \Asv(M)\geq\dim \AnV(M)=\dim \AnV(V)
\end{equation*}
\end{proof}

\begin{proof}[Proof of \Cref{thm:main}]
Denote $\cO:=\cO(V)$.
Since the number of nilpotent orbits is finite, $ \AnV(V)=\overline{\cO}$ is a finite union of nilpotent orbits: $\overline{\cO}=\cO\cup\bigcup_{i=1}^n \cO_i.$
Since $\fh\subset \fg$\ is spherical, Theorem \ref {thm:Geo} implies that for any $i$   we have
\begin{equation}\label{=Oi}
2\dim \cO_i\cap \fh^{\bot}\leq \dim\cO_i<\dim \cO\,\,.
\end{equation}
From (\ref{=Oi}) and \Cref{prop:GenMain} applied to $\fr=\fh$
 we obtain that $\cO$ intersects $\fh^{\bot}$
 and $$2\dim \cO\cap \fh^{\bot}\geq \dim \cO.$$
 By Theorem \ref {thm:Geo} again, we have  $2\dim \cO\cap \fh^{\bot}\leq \dim \cO$ and thus  $2\dim \cO\cap \fh^{\bot}=\dim \cO$.
\end{proof}

This proof was inspired by the proof of \cite[Theorem 8.4]{Vog}, which follows from Theorem \ref{thm:main} applied to a symmetric subgroup.
}

\begin{remark}\label{rem:Irr}
\DimaD{Since the irreducibility of $V$ was only used in the proof to show that the natural pairing is non-degenerate, we can weaken the irreducibility assumption and assume only that there exists $\xi\in \Hom_{\fh}(V,\fin^*)$ that does not vanish on any (non-trivial) $\fg$-submodule of $V$.}
\DimaE{In the conclusion of this more general theorem one will have to replace $\cO(V)$ by some orbit $\cO\subset \AnV(V)$ of maximal dimension.}
\end{remark}

\section{Proof of Theorem \ref{thm:maintwist}}\label{sec:twist}
\DimaA{
To prove Theorem \ref{thm:maintwist} we use the same ingredients as in the proof of Theorem \ref{thm:main}. However, we will need a different filtration on the universal enveloping algebra, that will be sensitive to the character $\chi$. In other words, we need $\fn$ to lie in non-positive filtras, and the derived algebra $[\fn,\fn]$ to lie in negative filtras. We will use the following Kazhdan filtration from the theory of $W$-algebras (\cite{GG}).

Fix a semi-simple element $h\in \fg$ such that $\ad(h)$ has integer eigenvalues.
Let $\cU_j$ denote the standard filtration on $\cU:=\cU(\fg)$, and let $\cU_j=\bigoplus_i\cU_j(i)$ denote the grading on $\cU_j$ given by the adjoint action of $h$.
Define the Kazhdan filtration on $\cU$ by
\begin{equation}
F^{\Kaz}_k\cU=\sum_{i+2j\leq k} \cU_j(i).
\end{equation}
Note that $k$ can be any integer - positive or negative, and that the filtras $F^{\Kaz}_k\cU$ are infinite-dimensional.
We have $\bigcap_kF^{\Kaz}_k\cU=\{0\}$ and $\bigcup_kF^{\Kaz}_k\cU=\cU$.
The associated graded algebra $S^{\Kaz}(\fg)$ is the symmetric algebra of $\fg$, where the grading is given on $\g(i)$ by  $i+2$.

For any finitely-generated $\fg$-module $V$, let
$\Anv^{\Kaz}(V)$ and $\Asv^{\Kaz}(V)$ denote the annihilator variety and the associated variety with respect to the Kazhdan filtration.
Let us show that $\Anv^{\Kaz}(V)$  coincides with the usual annihilator variety of $V$.
\begin{lem}
Let $I\subset \cU$ be a two-sided ideal, let $\gr(I)\subset S(\fg)$ denote the ideal spanned by the symbols of $I$ under the standard filtration, and $\gr^{\Kaz}(I)\subset S(\fg)$ denote the ideal spanned by the symbols of $I$ under the Kazhdan filtration. Then $\gr(I)=\gr^{\Kaz}(I)$.
\end{lem}
\begin{proof}
For any $s$, $\ad(h)$ acts semi-simply on $\cU_s$, and preserves the subspace $I\cap \cU_s$.  Thus, any element of $I\cap \cU_s$ is a sum of $\ad(h)$-homogeneous elements of $I\cap \cU_s$.
Since on $\ad(h)$-homogeneous elements the two symbol maps coincide, the images of $I$ in $S(\fg)$ under the two symbol maps span the same ideal of $S(\fg)$.

In more details: for any $s$, let $\sigma_s:\cU_s(\fg)\to \mathrm{S}_s(\fg)$ and $\kappa_s:F_s^{\Kaz}\cU\to \mathrm{S}_s^{\Kaz}(\fg)$ denote the symbol maps with respect to the standard and the Kazhdan filtration respectively.

Let us show first that $\gr(I)\subset \gr^{\Kaz}(I)$. By definition, $\gr(I)$ is spanned by elements of the form $\sigma_s(a)$ for some $a\in \cU_s\cap I$.
Since $I\cap \cU_s$ is an  $\ad(h)$-invariant subspace of $\cU_s$, we have
$a=\sum b_j$, where $b_j\in \cU_s(i_j)\cap I$ and $i_1<\dots<i_n\in \bZ$.
Let $k_j:=2s+i_{j}$. Then $\kappa_{k_j}(b_j)=\sigma_s(b_j)$. Thus
$\sigma_s(a)=\sum_{j}\kappa_{k_j}(b_j)\in \gr^{\Kaz}(I)$.

To show the opposite inclusion, let $a\in F_k^{\Kaz}(\cU)$. Let $s$ be minimal such that $a\in \cU_s$. Since $F_k^{\Kaz}(\cU)\cap  \cU_s\cap I$ is an $\ad(h)$-invariant subspace of $\cU_s$, we have $a=\sum b_j$, where
$$b_j\in \cU_s(i_{j})\cap F_k^{\Kaz}(\cU)\cap I=\cU_s(i_{j})\cap \cU_{l_{j}}\cap I,$$
where $l_j$ is the integer part of $(k-i_j)/2$.

Then $\kappa_k(b_j)=\sigma_{l_j}(b_j)$ for any $j$. Thus
$\kappa_k(a)=\sum \kappa_k(b_j)=\sum\sigma_{l_j}(b_j)\in \gr(I)$.
\end{proof}

\begin{cor}\label{cor:Kaz}
$\Anv^{\Kaz}(V)=\Anv(V)$.
\end{cor}

Now we can prove Theorem \ref{thm:maintwist} in a way similar to the proof of Theorem \ref{thm:main}. We first adapt the main steps.
By Gabber's theorem {\cite[Theorem I]{Gab}} we have

\begin{thm}[{cf. \cite[Theorem I]{Gab}}]\label{thm:TwiGab}
\DimaF{If $\Asv^{\Kaz}(V)$ is non-empty then}
$$\dim\Anv^{\Kaz}(V)\leq 2\dim\Asv^{\Kaz}(V).$$
\end{thm}

Now assume that all the eigenvalues of $\ad(h)$ are even, and let $\fn:=\bigoplus_{i<0}\fg(i)$.
Let $\chi \in \fg^*(2)$.
Let $\fs\subset \fg(0)$ be a spherical subgroup that stabilizes $\chi$, and let $\fh:=\fs\oplus \fn$. 
\DimaF{Denote also $\fn':=\bigoplus_{i<-2}\fg(i)$, and let $\cU(\fg)\fn'$ denote the left ideal generated by $\fn'$.

\begin{lem}\label{lem:non0}
$F^{\Kaz}_{-1}(\cU(\fg))\subset \cU(\fg)\fn'.$\end{lem}
\begin{proof}
 It is enough to prove  that for every $n$, $F^{\Kaz}_{-1}(\cU(\fg))\cap \cU_n(\fg)\subset \cU(\fg)\fn'$.
We prove this by induction on $n$. The base case $n=0$ is obvious, since $F^{\Kaz}_{-1}(\cU(\fg))\cap \cU_0(\fg)=0$. For the induction step let $a \in F^{\Kaz}_{-1}(\cU(\fg))\cap \cU_n(\fg)$. We can assume that $a$ is homogeneous with respect to the grading given by $ad(h)$, and let $d$ denote the degree. If $d\geq -2n$ then the condition $a \in F^{\Kaz}_{-1}(\cU(\fg))\cap \cU_n(\fg)$ implies $a\in \cU_{n-1}(\fg)$ and thus $a\in \cU(\fg)\fn'$ by the induction hypothesis. Thus from now on we assume that $d<-2n$.

For every $i$, fix a basis $\{X_i^j\}$ for $\fg(i)$. Compose from them a basis $B$ for $\fg$, ordered in the decreasing order of $i$. Write $a$ as a linear combination of PBW monomials $b_n^l$ corresponding to the  basis $B$.
Since $B$ consists of $ad(h)$-homogeneous elements, each monomial $b_n^l$ is also  $ad(h)$-homogeneous, of  degree $d$.
Each monomial $b_n^l$ is a product of  at most $n$  elements of $B=\{X_i^j\}$. Thus, in order to have degree $d<-2n$ with respect to $ad(h)$, it has to include an $X_i^j$ for  some $i<-2$. Because of the way we ordered $B$, the rightmost element of each $b_n^l$ must be an  $X_i^j$ for  some $i<-2$, and thus will lie in $\fn'$. Thus $b_n^l\in \cU(\fg)\fn'$ for each $l$, and thus $a\in \cU(\fg)\fn'$.
\end{proof}}


\begin{lemma}\label{lem:TwiAs}
Let $V$ be a $\fg$-module generated by a finite-dimensional
\DimaD{$\fh$-invariant vector subspace $W$ on which $\fn$ acts} via $\chi$.
 Then \DimaF{$\Asv^{\Kaz}(V)$ is a non-empty subset of} $\chi+\fh^{\bot}$.
\end{lemma}
\begin{proof}
\DimaD{Define $V_k:=F_k^{\Kaz}(\cU)W$ for any $k\in \bZ$. 
\DimaF{By Lemma \ref{lem:non0}, $V_{-1}=0$, thus $\gr_0(V)\neq 0$,  thus $\gr(V)\neq 0$ and $\Asv^{\Kaz}(V)$ is non-empty.
Now,} for any $i<0$ and any $X\in \fg(i)$} we have
$$(X-\chi(X))V_k=(X-\chi(X))F_k^{\Kaz}(\cU)W\subset [X,F_k^{\Kaz}(\cU)]W\subset F_{k+i}^{\Kaz}(\cU)W=V_{k+i}$$
Since $X\in F_{i+2}^{\Kaz}(\cU)$, and $\chi(X)=0$ unless $i=-2$, we obtain that $X-\chi(X)$ act by zero on $\gr^{\Kaz}(V)$.
\DimaD{Similarly, for any $X\in \fs$ we have
$$(X-\chi(X))V_k=XV_k=XF_k^{\Kaz}(\cU)W\subset [X,F_k^{\Kaz}(\cU)]W+F_k^{\Kaz}(\cU)XW\subset F_{k}^{\Kaz}(\cU)W=V_{k},$$
}
and thus again $X-\chi(X)$ acts by zero on $\gr^{\Kaz}(V)$.
 Thus $\Asv^{\Kaz}(V)\subset \chi+\fh^{\bot}$.
\end{proof}

\begin{lemma}\label{lem:dim}
For every nilpotent orbit $\cO\subset \fg^*$, we have $\dim (\overline{\cO}\cap (\chi+\fh^{\bot}))\leq \dim \cO/2$
\end{lemma}
\begin{proof}
Consider the restriction map $\overline{\cO}\to \fh^*$. Consider the fibers on the line through $\chi|_{\fh}\in \fh^*$. The dimension of the fiber of zero is at most $\dim \cO/2$ by Theorem \ref{thm:Geo}. All the other fibers are isomorphic using the action of the torus defined by $h$. Since the dimension of a special fiber is at least the dimension of the generic fiber, the lemma follows.
\end{proof}

We can now prove the following reformulation of Theorem \ref{thm:maintwist}.

\begin{thm}\label{thm:GenTwist}
Let $\fin\in \cM_{f.d.}(\fh)$ such that $\fn$ acts on $\fin$ via $\chi$.
Then for any $V\in \Irr(\fg)_{(\fh,\fin)}$ we have
\begin{equation}\label{=GT}
2\dim \left( \cO(V)\cap (\chi+\fh^{\bot})\right)=\dim \cO(V)
\end{equation}
In particular, the intersection is non-empty.
\end{thm}
\begin{proof}
Let $\xi:V\to \sigma^*$ be a non-zero $\fh$-equivariant linear map. Let $M\subset \Hom_{\C}(V,\sigma^*)$ be the $\fg$-module generated by $\xi$. As before, we have $\Ann(M)=\Ann(V)$, and thus $\Anv(M)=\overline{\cO}$, where $\cO:=\cO(V)$.
By Corollary \ref{cor:Kaz} and Theorem \ref{thm:TwiGab}, we have
\begin{equation}\label{=GTBer}
\dim \cO\leq 2\dim \Asv^{\Kaz}(M)
\end{equation}
On the other hand, by Lemma \ref{lem:TwiAs} we have
\begin{equation}\label{=GTInt}
\Asv^{\Kaz}(M)\subset \Anv(M) \cap (\chi+\fh^{\bot})=\overline{\cO}\cap (\chi+\fh^{\bot}).
\end{equation}
By Lemma \ref{lem:dim}, for every nilpotent orbit $\cO'$ we have
\begin{equation}\label{=GTIntDim}
2\dim (\cO'\cap (\chi+\fh^{\bot}))\leq \dim \cO'.
\end{equation}

From (\ref{=GTBer},\ref{=GTInt}, \ref{=GTIntDim}) we obtain \eqref{=GT}.
\end{proof}

\begin{remark}
A natural question to ask is whether $ \cO\cap (\chi+\fh^{\bot})$ is a Lagrangian subvariety of $\cO=\cO(V)$. While we do not know the answer to this question, \cite[Theorem I]{Gab} implies that at least one irreducible component of $\cO\cap (\chi+\fh^{\bot})$ of maximal dimension is a Lagrangian subvariety of $\cO$.
\end{remark}
}

\section{Casselman - Wallach representations and the proof of Corollary \ref{cor:CW}}\label{sec:CW}


Fix a real reductive group $G$, and a maximal compact subgroup $K \subset G$. For the definition of real reductive group see {\it e.g.} \cite[\S 2.1]{Wal}.
Let $\fg$ and $\mathfrak{k}$ denote the complexified Lie algebras of $G$ and $K$.

\begin{defn}
A $({\mathfrak{g}},K)$-module is a ${\mathfrak{g}}$-module $V$ with a
locally finite action of $K$ such the two induced actions of ${\mathfrak{k}}$
coincide and
$$\pi(Ad(k)(X))=\pi(k)\pi(X)\pi(k^{-1}) \text{ for any }k\in K \text{ and } X
\in {\mathfrak{g}}.$$

A finitely-generated $({\mathfrak{g}},K)$-module is called admissible if any
representation of $K$ appears in it with finite (or zero) multiplicity. In
this case we also call it a Harish-Chandra module.
\end{defn}

%
%
%
%
%

\begin{notation}
 Denote by $%
\mathcal{M}_{\infty}(G)$ the category of smooth admissible {Fr\'{e}chet \,}
representations of $G$ of moderate growth (see \cite[\S 11.5]{Wal} or \cite%
{CasGlob}), where admissible means that the space of $K$-finite vectors is
an admissible $({\mathfrak{g}},K)$-module. Denote by $\mathcal{M}_{HC}(G)$
the category of admissible $({\mathfrak{g}},K)$-modules. Note that both $%
\mathcal{M}_{\infty}(G)$ and $\mathcal{M}_{HC}(G)$ are naturally
subcategories of $\mathcal{M}({\mathfrak{g}})$. We denote by $HC:\mathcal{M}%
_\infty(G) \to \mathcal{M}_{HC}(G)$ the functor of $K$-finite vectors.
Denote the collection of irreducible objects in $\cM(G)$ by $\Irr(G)$.
\end{notation}

 It is well known that for any $\pi\in \cM_{\infty}(G),$ the $(\fg,K)$-submodule  $HC(\pi)$ is dense in $\pi$.

\begin{thm}[{Casselman-Wallach, see \cite[\S 11.6.8]{Wal},  or \cite{CasGlob}}]\label{thm:CW}
$\,$\\ The functor $HC:\mathcal{M}_{\infty}(G) \to \mathcal{M}%
_{HC}(G)$ is an equivalence of categories.
\end{thm}

For $M\in \cM(\fg)$ and $g\in G$, denote by $M^g\in \cM(\fg)$ the module that coincides with $M$ as a vector space, but the action of $\fg$ on it is twisted by the adjoint action of $g$ on $\fg$.

\begin{lem}\label{lem:IrrHC}
For any irreducible Harish-Chandra module $V$, there exist an irreducible $\fg$-submodule $M\subset V$ and a finite set $S\subset K$ such that $V$ is a quotient of $\bigoplus_{k\in S}M^k$ as a $\fg$-module, and $\AnV(V)=\ \AnV(M)=\overline{\cO(M)}$.
\end{lem}
\begin{proof}
In the proof we will use the functor of admissible dual - {\it i.e.} the space of $K$-finite vectors in the dual space. For every Harish-Chandra module $V$, natural map $V\to \widetilde{\widetilde{V}}$ is an isomorphism.

Let $\widetilde{V}$ be the admissible dual of $V$.
Since it is irreducible as a $(\fg,K)$-module, it is finitely-generated, and thus Noetherian, as a $\fg$-module. Thus, as a $\fg$-module, it has an irreducible quotient $Y$. Let $S$ be a set of representatives of the connected components of $K$. Then we have a natural map $\widetilde{V}\to \bigoplus_{k\in S} Y^k$. The kernel of this map is a proper $(\fg,K)$-submodule of $\widetilde{V}$ and thus is zero. Let $M:=\widetilde{Y}$. Then $M$ is a $\fg$-submodule of $V$, and the embedding $\widetilde{V}\to \bigoplus_{k\in S} Y^k$ defines an epimorphism $\bigoplus_{k\in S} M^k\onto V$.

Since annihilator varieties are invariant with respect to the adjoint action, we also have $\AnV(M^k)=\AnV(M)$ for any $k$, and thus $\AnV(M)=\AnV(V)$.
\end{proof}

\begin{cor}
For any $\pi\in \Irr(G)$, there exists a (unique) nilpotent orbit $\cO$ such that $\AnV(\pi)=\overline{\cO}$. We will denote $\cO(\pi):=\cO$.
\end{cor}
\begin{proof}
Let $V$ be the Harish-Chandra module of $\pi$. Then $V$ is irreducible and thus $\AnV(V)=\overline{\cO}$ for some (unique) $\cO$. Since $V$ is dense in $\pi$, $\AnV(V)=\AnV(\pi)$.
\end{proof}

\subsection{Proof of Corollary \ref{cor:CW}}\label{subsec:PfB}

Let
$\pi\in \Irr(G)$ and let $\xi:\pi\to \sigma$ be a non-zero continuous $\fh$-equivariant linear map. By Theorem \ref{thm:CW}, the Harish-Chandra module $V:=HC(\pi)$ is also irreducible.  Since $V$ is dense in $\pi$, the restriction $\xi|_V$ does not vanish. Let $M\subset V$ be as in Lemma \ref{lem:IrrHC} above. Then $\xi|_{M^k}\neq 0$ for some $k\in K$, and thus $M^k\in \Irr(\fg)_{(\fh,\sigma^*)}$.
Finally, we have $\cO(\pi)=\cO(V)=\cO(M^k)$. Thus Corollary \ref{cor:CW} follows from Theorems \ref{thm:main} and \ref{thm:maintwist}.
\proofend


\section{First applications }\label{sec:appl}

\subsection{Symmetric pairs}\label{subsec:Pra}
\DimaB{
One can apply Corollary \ref{cor:CW} to symmetric pairs $H\subset G$, since for them $\cN(\fg^*)\cap {\fh}^{\bot}$ is classified, see
\cite{Oht,Djo1,Djo2}. \Dima{The motivation for this classification is the Cartan classification, {\it i.e.} a correspondence between conjugacy classes of involutions of $\bf G$, and isomorphism classes of real structures on $\bf G$, cf. \cite[Theorem 6 of Chapter III, \S 4]{SerreGal} or \cite[Theorem 3.2]{AdC}. The correspondence is defined in the following way.

Let $\theta$ be an involution of $\bf G$ defined over $\R$.
Then the real reductive group $G_H$ corresponding to $\theta$
 is the group of fixed points in the complex group ${\bfG}$ of the anti-holomorphic involution $\sigma$ obtained by composing $\theta$ with complex conjugation.
Denote by  $\bf H=G^{\theta}$ be the group of fixed points of $\theta$.
Then $\bfH \cap G_H $ is the maximal compact subgroup of $G_H$.

Now let $\fg_H\subset \fg$ denote the \emph{real} Lie algebra of $G_H$.
Let $\cO\subset \cN(\fg^*)$  be a nilpotent orbit. Then there exists the Kostant-Sekiguchi correspondence \cite{Sek} between $\bf H$-orbits in $\cO\cap {\fh}^{\bot}$ and $G_H$-orbits in $\cO\cap (\fg_H)^*$. In particular, $\cO$ intersects ${\fh}^{\bot}$ if and only if  $\cO$ intersects the space of ``real points" $(\fg_H)^*$.
Thus, from Corollary \ref{cor:CW} we obtain the following statement.

\begin{cor}\label{cor:Sym}
For any $\sigma \in \cM_{f.d.}(\fh)$ and any $\pi\in \Irr(G)_{\fh,\sigma},$ the orbit $\cO(\pi)$ intersects $(\fg_H)^*$.
\end{cor}

This statement was inspired by \cite{PraSak}, which shows that over non-Archimedean fields, there exists an $[H,H]$-distinguished} generic irreducible representation if and only if $G_H$ is quasi-split. For real reductive groups this statement is \cite[Corollary F]{GS_X}.

Prasad also asks whether  if $G$ is quasi-split, $H$ is split and $\Irr(G)_H$ includes a tempered representation then necessarily $\Irr(G)_H$ includes a generic representation - see  \cite[Question 1]{PraSak}\footnote{\Dima{\cite[Question 1]{PraSak} asks further whether for any $H$-distinguished tempered representation, there exists an $H$-distinguished generic representation in the same $L$-packet. We can say nothing about that.}}.
Using Corollary \ref{cor:Sym},  we can answer this question in the affirmative in several cases in the following way.

}

By \cite{Harris}, if $\pi$ is tempered then its wave-front set $\WF(\pi)$ consists of closures of $\R$-distinguished $G$-orbits in the real points of the nilpotent \Dima{cone} of $\fg^*$. The $\R$-distinguished orbits are orbits for which the reductive Levi
factor of the centralizer of any element  is a compact Lie algebra.
They are classified in \cite[Theorems 8-14]{PT} under the name compact orbits.
By \cite[Theorem D]{Ross}, the maximal $G$-orbits in $\WF(\pi)$ lie in $\cO(\pi)$, and thus we obtain that for tempered $\pi$, $\cO(\pi)$ includes an $\R$-distinguished real orbit.
For exceptional pairs, \Dima{Corollary \ref{cor:Sym}} implies that the pairs corresponding to the real forms $F_{4(-20)}, E_{6(-26)},$ and $E_{7(-25)}$ have no distinguished tempered representations. This answers the question affirmatively for $G=F_4$, leaving only the exceptional pairs corresponding to the real forms $E_{6(-14)}$, $E_{7(-25)}$, $E_{8(-24)}$.

For classical symmetric pairs, we obtain an affirmative answer to  Prasad's question for all cases except when $G$ is a unitary group, or when the symmetric pair $(G,H)$ corresponds to one of the real reductive groups $O^*(2n),\, \Sp(n,n),$ or $O(m,n)$ with $|m-n|>1$.

To summarize, Corollary \ref{cor:Sym} and the papers mentioned above 
imply the next  proposition.

\begin{prop}
Suppose that $G$ is quasi-split, and that $\Irr(G)_{(\fh,\sigma)}$ includes a tempered representation, for some \Dima{finite-dimensional representation $\sigma$} of $\fh$. Then,  every simple component \Dima{$ (G',H')$} of the symmetric pair $(G,H)$ satisfies at least one of the following:
\begin{enumerate}[(a)]
\item The real reductive group \Dima{$G'_{H'}$}  is quasi-split.
\item The real reductive group \Dima{$G'_{H'}$}  is one of the following
\begin{equation}
O^*(2n), \quad \Sp(n,n), \quad O(m,n)\text { with }|m-n|>1, \quad E_{6(-14)},\quad  E_{7(-25)}, \quad E_{8(-24)}
\end{equation}
\item $\Dima{G'}\cong SU(p,p)$ or  $\Dima{G'}\cong SU(p,p+1)$ for some $p$.
\end{enumerate}
\end{prop}

\subsection{Branching problems and translation functors}\label{subsec:bra}

Let $H\subset G$ be a reductive subgroup such that $\Delta H \subset H\times G$ is a spherical subgroup. Let $\pi\in \Irr(G)$ and $\tau\in \Irr(H)$. We are interested in the non-vanishing of the multiplicity space $\Hom_H(\pi|_H,\tau)$. We have $$\Hom_H(\pi|_H,\tau)\cong ((\pi\hot \widetilde{\tau})^*)^{\Delta H},$$
where $\widetilde{\tau}$ denotes the contragredient representation, and $\hot$ denotes the completed projective tensor product.
 Let  $p_{\fh}:\fg^*\onto \fh^*$ denote the natural projection.
Since $\cO(\tau)=\cO(\widetilde{\tau})$, Corollary \ref{cor:CW} applied to the pair $\Delta H \subset H\times G$ implies the following statement.

\begin{cor}\label{cor:strong}
Let $\pi\in \Irr(G)$, $\tau\in \Irr(H)$, and $\fin\in \cM_{f.d.}(H)$. Consider the diagonal action of $H$ on $\tau\otimes_{\C} \fin$.
If $\Hom_{H}(\pi|_H,\tau\otimes_{\C} \fin)\neq 0$ then $\cO(\tau)\subset p_{\fh}(\cO(\pi))$.
\end{cor}

Such pairs (under the additional assumption that $H$ is a symmetric subgroup of $G$) are classified in \cite{KM}. They include the multiplicity one pairs corresponding to the case $k=1$ of Theorem \ref{thm:strongPart}. One can deduce this case of the theorem from Corollary \ref{cor:strong}, but instead we will treat the general case in \S \ref{sec:TwistAppl}.

Setting $G=H$ in Corollary \ref{cor:strong} we obtain the following corollary on annihilator varieties of translation functors (see \cite[Definition 4.5.7]{VogBook}).

\begin{cor}\label{cor:trans}
Let $\tau\in \Irr(G)$, and let $\psi_{\lambda}^{\lambda'}(\tau)$ be a translation of $\tau$.  Then for any irreducible submodule $\pi\subset \psi_{\lambda}^{\lambda'}(\tau)$, we have $\cO(\pi)=\cO(\tau)$.
\end{cor}
Dualizing, we obtain the same equality for any irreducible quotient of $\psi_{\lambda}^{\lambda'}(\tau)$.

\begin{remark}
Corollary \ref{cor:strong} complies with the philosophy of the Kirillov-Kostant orbit method. Indeed, an analogous statement for nilpotent groups was proven in \cite{CorGr}. For holomorphic discrete series representations of reductive groups, an analogous statement is proven in \cite{Par}.
\end{remark}

%

\subsection{Annihilator varieties of Jacquet quotients}\label{subsec:Jac}

Let $P\subset G$ be a parabolic subgroup, $N$ be its unipotent radical and $L:=P/N$. Let $\pi\in \Irr(G)$, and let $r_N(\pi)\in \cM_{\infty}(L)$ be its Jacquet reduction.
Corollary \ref{cor:CW} gives a restriction on quotients of $r_N(\pi)$ by considering the subgroup
\begin{equation}\label{=Jac}
H:=\{(pN,p)\in L\times G \, : \, p\in P\}\subset L\times G
\end{equation}
In this case, we have $\fh^{\bot}=\{(-\mu|_{\fp},\mu)\, \vert \mu\in \fn^{\bot}\}$
and $p_{\fp}(\fn^{\bot})\cong \fl^*$, where
$p_{\fp}:\fg^*\onto \fp^*$ the restriction  to $\fp$.  In addition, for any $\tau\in \Irr(L)$ we have

$$((\widetilde{\tau}\hot \pi)^*)^H\cong ((\widetilde{\tau}\hot r_N(\pi))^*)^{\Delta (L)}=\Hom_{L}(r_N(\pi),\tau).$$
Thus, Corollary \ref{cor:CW} implies the following corollary.
\begin{cor}\label{cor:J}
For any $\pi\in \Irr(G)$ and any irreducible quotient \Dima{$\tau$} of $r_P(\pi)$ we have
$$\cO(\Dima{\tau})\subset p_{\fp}(\cO(\pi)\cap \fn^{\bot})$$
\end{cor}

We remark that the computation of the set $p_{\fp}(\cO(\pi)\cap \fn^{\bot})$ is an interesting new geometric question. One can reformulate it in a dual way, by identifying
$\fg\cong \fg^*, \, \fn^{\bot} \cong \fp$ and $\fl\cong \fl^*$ using the Killing form, and embedding $\fl$ into $\fp$.
Then the question becomes:
given a nilpotent orbit $\cO\subset \fl$, what nilpotent orbits in $\fg$ can we hit by adding elements of $\fn$. The maximal among these orbits is the induced orbit, and the minimal is the saturation of $\cO$, but already in the case of $\fg=\gl_n$ we see from Lemma \ref{lem:quot} below that not every orbit in between is possible.

\Dima{
Using this reformulation and the Frobenius reciprocity, we obtain the following corollary.

\begin{cor}
Let $\tau \in \Irr(L)$. Then, for any irreducible subrepresentation $\pi\subset \Ind_P^G(\tau)$, the orbit $\cO(\pi)$ intersects $\cO(\Dima{\tau})+\fn$. In particular, for any  subrepresentation $\pi'\subset \Ind_P^G(\tau)$, we have $\cO(\tau)\subset \AnV(\pi')$.
\end{cor}
Passing to admissible dual, we get the same statement about quotients of parabolic inductions. This corollary is compatible with Ginzburg's conjectures on orbits corresponding to residues of Eisenstein series \cite{GinzConj}.
}

\DimaB{In the case $\fg=\gl_n$ and $\fp\subset \gl_n$ a maximal parabolic subgroup, the saturation of $\cO+\fn$ is computed in \cite{Zha} as follows.
In this case $\fl\cong\gl_k\times \gl_{n-k}$, and thus every orbit $\cO\subset \fl$ is given by two partitions: $\mu$ of $k$  and $\nu$ of $n-k$. Let $\lambda$ be a partition of $n$, and $\cO'_{\lambda}\subset \gl_n$ be the corresponding orbit. By \cite{Zha}, $\cO'_{\lambda}$ intersects $\cO+\fn$ if and only if the Littlewood-Richardson coefficient $c^{\lambda}_{\mu,\nu}$ does not vanish.
 }

\DimaA{
Let us now prove Corollary \ref{cor:TJ}, which is a twisted version of Corollary \ref{cor:J}. In this corollary we fix a Levi decomposition $P=LN$, a reductive subgroup $S\subset L$ such that $\Delta S \subset S\times L$ is a spherical subgroup, and a character $\chi$ of $\fr=\fs\oplus \fn$ that is trivial on $\fs$.
We fix $\tau\in \Irr(S)$ and $\pi\in \Irr(G)$ such that $\Hom_{\fr}(\pi,\tau\otimes \chi)\neq 0$ and have to show that
$\cO(\Dima{\tau})$ intersects $\cO(\pi)|_{\fr}+\chi$.
\begin{proof}[Proof of Corollary \ref{cor:TJ}]
Let $R:=SN$. Let
\begin{equation*}
Y:=\{(pN,p)\in S\times R\}\subset S\times G,
\end{equation*}
and $\mathfrak{y}$ be the Lie algebra of $Y$. The restriction to $\fr$ defines a map $\fn^{\bot}\onto \fr^*$. We have $$\mathfrak{y}^{\bot}=\{-\mu|_{\fr},\mu)\, \vert \mu\in \fn^{\bot}\}.$$ We also have
$$\Hom_{\fr}(\pi,\tau\otimes \chi)\cong \Hom_{\fy}(\widetilde{\tau}\hot \pi,\chi)$$
Thus, the statement follows from Corollary \ref{cor:CW}.
\end{proof}

}

\section{Applications to mixed models}\label{sec:TwistAppl}
In this section we fix $F$ to denote either $\R$ or $\C$.

\subsection{Preliminaries on nilpotent orbits and partitions}

By \cite[\S 5.1]{CM}, if $G$ is a classical real reductive group then any nilpotent orbit  $\cO\subset \fg^*$ is given by a partition $\lambda(\cO)$.
 For orthogonal groups, the restriction on $\lambda$ is that even parts have even multiplicities. To very even partitions, {\it i.e.} partitions consisting of even parts only, there correspond two orbits of $SO_n(\C)$, but their union is a single orbit of $O_n(\C)$.

 For a partition $\lambda$ we denote by $\lambda^t$ the transposed partition. We denote by $\lambda^t_i$ the $i$-th part of $\lambda^t$ if $i\leq \lambda_1$ (which is the length of $\lambda^t$) and set $\lambda^t_i$ to zero for $i>\lambda_1$.

For $G=U(m,n)$, the real nilpotent orbits are described in \cite[\S 9.3]{CM}. From this description we see that the number of odd parts in the partition of the corresponding complex nilpotent orbit  is  $|m-n|+2l$ for some $k\geq 0\in \bZ$. This is relevant since for any $\pi\in \Irr(G)$, $\cO(\pi)$ includes a real nilpotent orbit for $G$ by \cite[Theorem D]{Ross}.

If $G$ is a complex reductive group, regarded as a real group
then the real Lie algebra ${\mathfrak{g}}_{0}$ of $G$ is already a complex Lie
algebra, and we have ${\mathfrak{g}}\cong{\mathfrak{g}}_{0}\times
{\mathfrak{g}}_{0}$, and ${\mathfrak{g}}_{0}$ is diagonally embedded into
${\mathfrak{g}}$. The Lie algebra ${\mathfrak{k}}$ is also isomorphic to
${\mathfrak{g}}_{0}$, and is embedded into ${\mathfrak{g}}$ by $X\mapsto
(X,\theta\left(  X\right)  )$, where $\theta$ is the Cartan involution of $G$, {\it i.e.} the involution such that $K=G^{\theta}$. For a nilpotent orbit ${\mathcal{O}}%
\subset\mathcal{N}({\mathfrak{g}})$ we have ${\mathcal{O}}={\mathcal{O}}%
_{1}\times{\mathcal{O}}_{2}$ where ${\mathcal{O}}_{i}\subset\mathcal{N}%
({\mathfrak{g}}_{0})$. However, if ${\mathcal{O}}$ intersects  ${\mathfrak{k}}^{\bot}\subset
{\mathfrak{g}^{*}}$ then ${\mathcal{O}}_{1}={\mathcal{O}}_{2}$, and thus
${\mathcal{O}}$ is defined by a single nilpotent orbit in ${\mathfrak{g}}_{0}%
$. By \cite[Theorem 8.4]{Vog}, $\cO(V)$ intersects $\fk^{\bot}$
for any $V\in \cM_{HC}(G)$, and thus we will be only interested in
such orbits.
Therefore, if $G$ is a  classical complex reductive group then any nilpotent orbit  $\cO\subset \fg^*$ is still described by a single partition.

Thus, if $G$ is a real or complex classical group and $\pi\in \Irr(G)$ then to the orbit $\cO(\pi)$ there corresponds a single partition, that we will denote by  $\lambda(\pi)$.

 \subsection{Rankin-Selberg models}\label{subsec:RS}
\DimaA{
Let \Dima{$S:=\GL_{n}(F)$, $G:=\GL_{n+k}(F)$ for some $k\geq 1$,} and consider the embedding of $S$ into the upper left corner of $G$. Let
$P\subset G$ be the standard block upper-triangular parabolic subgroup with blocks $n\Dima{+}1,1,\dots,1$.
We identify $\fg$ with $\fg^*$ using the trace form pairing.
Let
$\phi:N\to F$ be the unitary character with differential $\chi=d\phi$  given by the matrix $A=\sum_{i=n+1}^{n+k-1}E_{i+1,i}$, where $E_{ij}$ denote elementary matrices. \Dima{If $k=1$ we take $A=0$. }Let $R:=SN$ and extend $\phi$ to a character of $R$ trivial on $S$.


The standard Levi subgroup of $P$ is $L=\GL_{n+1}(F)\times (F^{\times})^{k-1}$, and  $\Delta S$ is a spherical subgroup of $L\times S$ by \cite{KM}. Thus this case of Theorem \ref{thm:strongPart} follows from Corollary \ref{cor:TJ} and the following proposition.

\begin{prop}\label{prop:RS}
Let $T\in \cN(\gl_{n})$ and $X\in A+T+\fr^{\bot}\in \cN(\gl_{n+k})$, and let $\lambda(T)$ and $\lambda(X)$ be the corresponding partitions. Then for any index $i$ we have
\begin{equation}\label{=RS}
|\lambda(T)^t_i-\lambda(X)^t_i|\leq 1
\end{equation}
\end{prop}

For the proof we will need the following lemma.

\begin{lemma}[Alisa Qui]\label{lem:quot}
Let $V$ be a vector space, and $Y$ be a nilpotent operator on $V$. Let $v\in V$, and  let $V_1:=\Span\{Y^i v\}_{i\geq 0}$. Let $V':=V/V_1$ and let $Y'$ be the operator on $V'$ given by $Y$. Let $\nu$ and $\nu'$ be the partitions corresponding to $Y$ and $Y'$. Then for any $i\geq 1$, we have $$\nu^t_i\in \{(\nu')^t_i,(\nu')^t_i+1\}.$$
\end{lemma}
\begin{proof}
Let $p$ be the monic polynomial of minimal degree such that $p(Y)v=0$. Then $p$ divides the minimal polynomial of $Y$, and thus $p$ is a monomial.
Thus all the non-zero vectors in the set  $\{Y^i v\}_{i\geq 0}$ are linear independent.

For every $i\geq 0$ let $a_i$ denote the minimal index $j$ such that $Y^jv\in \Im Y^i$. Then we have
\begin{equation}
\rk (Y')^i=\rk Y^i-\dim V_1+a_i \quad \text{and} \quad a_i\geq a_{i+1}\geq a_i + 1\,\,.
\end{equation}
We also have
\begin{equation}
\nu^t_i=\rk Y^{i-1}-\rk Y^i\quad \text{and} \quad(\nu')^t_i=\rk (Y')^{i-1}-\rk (Y')^i\,\,.
\end{equation}
From these equations we obtain
\begin{equation}
\nu^t_i-(\nu')^t_i=\rk Y^{i-1}-\rk (Y')^{i-1}-(\rk Y^i-\rk (Y')^i)=a_{i-1}-a_i\in\{0,1\}
\end{equation}
\end{proof}

\begin{proof}[Proof of Proposition \ref{prop:RS}]
It is easy to see that $T$ is conjugate to a matrix of the form $D=X+A+B+C$, where $A=\sum_{i=n+1}^{n+k-1}E_{i+1,i}$ as above, $B_{ij}=0$ unless $i=n+1,$ and $j\leq n$, and $C_{ij}=0$ unless $j=n+k$. From $\Tr(X)=\Tr(T)=0$ we have $C_{n+k,n+k}=0$. For any $i\leq n+k$, let $e_{i}$ denote the $i$-th vector in the standard basis. Let $V$ be the span of $e_{1},\dots,e_n$,  let $v\in V$ be the vector given by $v_i=C_{i,n+k}$ and $V_1:=\Span\{X^iv\, \vert i=1\dots n\}$. Let $W:=\Span\{e_1,\dots, e_{n+k}\}$, and $w:=e_{n+1}$. Then $W_1=\Span\{D^iw\}=V_1\oplus \Span\{e_{n+1},\dots, e_{n+k}\}$. Let $D'$ and $X'$ denote the operators given by $D$ and $X$
on $W':=W/W_1$ and $V':=V/V_1$. Then $X'=D'$ under the natural
isomorphism $W'\cong V'$. Let $\nu'$ denote the partition corresponding to $X'$. Then, by Lemma \ref{lem:quot} applied to $V$ and to $W$, for any index $i$ we have  $\lambda^t_i(T),\lambda^t_i(X)\in \{(\nu')^t, (\nu')^t+1\}$. Thus $|\lambda(T)^t_i-\lambda(X)^t_i|\leq 1$.
\end{proof}


For representations of Arthur type, \cite[Conjecture 5.1]{GGP} formulates conjectural natural and sufficient conditions for the existence of Rankin-Selberg models, in terms of the Arthur parameters of $\pi$ and $\tau$.

\begin{remark}
In the same way one proves the analogous theorem for $\SL_{n+k}(F)$.
\end{remark}



\subsection{Bessel models}\label{subsec:Bessel}
Let $V$ be a non-degenerate quadratic or Hermitian vector space over $F$, and let $G$ be the group of symmetries of $V$ ({\it i.e.} $G=\mathrm{O}(V)$ or $G=\mathrm{SO}(V)$ or  $G=\mathrm{U}(V)$ or $G=\mathrm{SU}(V)$). Let $W\subset V$ be an isotropic subspace\Dima{. Let $k:=\dim W+1$.} Choose a full flag $F_1\subset \dots \subset F_{k-1}=W$ in $W$, and let $P\subset G$ be its stabilizer, and $N$ be the unipotent radical of $P$. Let $U\subset W^{\bot}\subset V$ be a maximal non-degenerate subspace, and let $L$ be the group of symmetries of $V$ . Then $L$ is a Levi subgroup of $P$. Choose an anisotropic line $X\subset U$, and let $S$ be the stabilizer of $X$ in $L$.  To define a character $\phi$ of $N$, let $F_{k}:=W\oplus X$, project $\Dima{N}$ on $\times_{i=1}^{k-1}\Hom(F_{i+1}/F_i,F_{i}/F_{i-1})$, and let $\phi$ be a product of non-trivial \Dima{unitary} characters of the coordinates. \Dima{Note that for $k=1$, $N$ is the trivial group.} Let $R:=SN$, and extend $\chi$ to the character of $R$ that is trivial on $S$.

Since $\Delta S$ is a spherical subgroup of $L\times S$ by \cite{KM}, the orthogonal and unitary cases of Theorem \ref{thm:strongPart}  follow from Corollary \ref{cor:TJ} and the following proposition, that is proven similarly to Proposition \ref{prop:RS}.
\begin{prop}\label{prop:Bes}
Let $T\in \cN(\fs^*)$ and $X\in \chi+T+\fr^{\bot}\subset \cN(\g^*)$, and let $\lambda(T)$ and $\lambda(X)$ be the corresponding partitions. Then for any index $i$ we have
\begin{equation}
|\lambda(T)^t_i-\lambda(X)^t_i|\leq 1
\end{equation}
\end{prop}


\begin{remark}
The group $\mathrm{O}(V)$ is not Zariski connected, and thus Corollary \ref{cor:CW} is not directly applicable to its representations.  However, the restriction of any such representation $\pi$ to $\mathrm{SO}(V)$ is composed of at most two irreducible components, intertwined by the component group $\bZ/2\bZ$. Thus, the two components either have the same annihilator variety, or they have annihilator varieties corresponding to two nilpotent orbits described by the same  partition. Also, the invariant functional has non-zero restriction to at least one of the two components. Thus, the special orthogonal case of Theorem \ref{thm:strongPart} implies the  orthogonal case.
\end{remark}
}

\Dima{
\subsection{Archimedean analogs of Bernstein-Zelevinsky derivatives}\label{subsec:Der}

In \cite{AGS}, two Archimedean analogues of the Bernstein-Zelevinsky derivatives (\cite{BZ}) are defined. One functor, $B^k$, is a more naive analogue, and preserves admissibility. Another one, $E^k$, does not in general preserve admissibility, but is exact. We refer the reader to \cite{AGS} for the precise definitions. Corollary \ref{cor:TJ}, Lemma \ref{lem:quot} and Theorem \ref{thm:strongPart} give the following description of orbits associated to irreducible quotients of derivatives.

\begin{prop}
Let $\pi\in \Irr(\GL_{n+k}(F))$, let $\tau\in \Irr(\GL_n(F))$ be an irreducible quotient of $B^k(\pi)$ and $\rho \in \Irr(\GL_n(F))$ be an irreducible quotient of $E^k(\pi)$. Then for any index $i$ we have
$$0\leq \lambda_i^t(\pi) - \lambda_i^t(\tau)\leq 1\quad \text{and}\quad -1\leq \lambda^t_i(\pi) - \lambda_i^t(\rho)\leq 1$$
\end{prop}
}

\subsection{Klyachko models} \label{subsec:Klya}

Let $G:=\GL_{2n+k}(F)$ with $k\geq 0$. For $k=0$ we let $H:=\Sp_{2n}\subset G$ and let $\phi$ be the trivial character of $H$. In this case, the nilpotent matrices in $\mathfrak{sp}_{2n}^\bot$ have very even partitions. This follows from \cite{Oht}, and can also be checked by direct computation.

Assume now that $k>0$ we let $P=LN\subset G$ be the standard block upper-triangular parabolic subgroups with blocks $2n+1,1,\dots,1$. Let $S$ be the group of matrices of the form $\begin{pmatrix}g & h \\
0 & 1 \\
\end{pmatrix}$ with $g\in\Sp_{2n}(F)$, embedded in the first block of $L$.
As in \S \ref{subsec:RS} we identify $\fg$ with $\fg^*$ using the trace form pairing and let
$\phi$ be the unitary character of $N$ with differential $\chi=d\phi$  given by the matrix $A=\sum_{i=2n+1}^{2n+k-1}E_{i+1,i}$. Then we have
$$\chi+{\fh}^{\bot}=A+\left \{\begin{pmatrix}B & C \\
0 & D \\
\end{pmatrix} \in \gl_{2n+k}\, \vert \, B\bot \mathfrak{sp}_{2n} \text{ and }D \text{ is upper triangular }\right\}.$$
Let $X=\begin{pmatrix}B & C \\
0 & E \\
\end{pmatrix}\in \chi+{\fh}^{\bot}$ be nilpotent. Then $E$ is a regular nilpotent $k$ by $k$ matrix, and $\lambda(B)^t$ is very even. By Lemma \ref{lem:quot},
this implies that $\lambda(X)^t$ has exactly $k$ odd parts.

Summarizing, we obtain that for any $k\geq 0$, Corollary \ref{cor:CW} implies the following one.
\begin{cor}\label{cor:Klya}
Let $\pi\in \Irr(G)$ with $\Hom_{H}(\pi|_{H}, \phi)\neq 0$. Then $\lambda(\pi)^t$ has exactly $k$ odd parts.
\end{cor}
For unitarizable $\pi$,  this corollary and the converse statement follow from \cite{AOS,GOSS}.

\subsection{Shalika models}\label{subsec:Sha}

Let $G:=\GL_{2n}(F)$, and
$P\subset G$ be the standard block upper-triangular parabolic subgroups with blocks of size $n$. Let $S:=\GL_{n}(F)$, embedded diagonally in $L=\GL_{n}(F)\times \GL_{n}(F)\subset P$. We identify the nilpotent radical $\fn$ of $\fp$ with $\gl_n(F)$ (sitting in the upper-right corner of $\fg$), and let
$\phi:N\to F$ such that $\chi=d\phi$ is given by trace. Let $H=SN$, and extend $\phi$ to a character of $H$. Then we have
$$\chi+\fh^{\bot}=\left \{\begin{pmatrix}B & C \\
\id & -B \\
\end{pmatrix} \in \gl_{2n}\right\}$$
Thus any $X\in \chi+\fh^{\bot}$ is conjugate to $Y=\begin{pmatrix}0 & C \\
\id & 0 \\
\end{pmatrix}$, and is nilpotent if and only if $C$ is nilpotent.
Further, for any $i>0$ we have $\rk X^{2i+1}-\rk X^{2i+2}=\rk C^{i}-\rk C^{i+1}=\rk X^{2i}-\rk X^{2i+1}$. Thus $\lambda(X)^t_{2i+1}=\lambda(X)^t_{2i+2}=\lambda(C)^t_{i+1}$.
That is, $\lambda^t$ has even multiplicities, or equivalently $\lambda$ is very even. Thus, Corollary \ref{cor:CW} implies the following one.
\begin{cor}
Let $\pi\in \Irr(G)$ with $\Hom_{H}(\pi|_{H}, \phi)\neq 0$. Then $\lambda^t(\pi)$ has even multiplicities.
\end{cor}
Curiously, this condition is dual to the condition on $\pi$ to be distinguished under the symplectic subgroup, which requires $\lambda^t(\pi)$ to be very even.

\Dima{
\subsection{Ginzburg-Rallis models} Let $G:=\GL_6(F)$ and let $L\cong\GL_2(F)\times \GL_2(F)\times \GL_2(F)$ be a standard Levi subgroup. Let $(H,\chi)$ be a Whittaker induction of the spherical pair $\Delta\GL_2(F)\subset L$.
Then Corollary \ref{cor:CW} implies that for any $\pi \in \Irr(G)_{\fh,\chi}$, $\lambda^t(\pi)\in \{1^6,21^4,2^3\}$.}

\section{Non-homogeneous spherical varieties}\label{sec:X}
Let $\bf X$ be a spherical smooth $\bfG$-variety defined over $\R$, and let $\mu_{\bf X}: T^*\bf X \to \fg^*$ denote the moment map defined in the introduction.
Let $X$ be a union of connected components of the manifold of real points of $\bf X$. Let $\cE$ be an algebraic bundle over $X$, and let $\Sc(X,\cE)$ denote the space Schwartz sections of $\cE$ (see \S \ref{subsec:Sc} below for the definition).
\begin{thm}\label{thm:X}
If $\pi\in \Irr(G)$ is a quotient of  $\Sc(X,\cE)$ then $\cO(\pi)\subset \Im(\mu_{\bf X})$.\end{thm}

We prove this theorem by induction on the number of $\bf G$-orbits on $\bf X$. The base case of a single orbit follows from Corollary \ref{cor:CW} by  Frobenius reciprocity for Schwartz induction (Theorem \ref{thm:Frob} below). The induction step will follow from the theory of Schwartz functions on algebraic manifolds. We will now give the necessary preliminaries on nuclear \Fre$\,$ spaces, Schwartz functions on real algebraic manifolds, and then prove the theorem in \S \ref{subsec:PfX}.

In \S \ref{subsec:PfAltX} we derive Theorem \ref{thm:AltX} from Theorem \ref{thm:X}, the Casselman embedding theorem and \cite[Theorem C]{AGKL}.

\begin{remark}
In the homogeneous case $\bf X=G/H$, Corollary \ref{cor:CW} and Remark \ref{rem:Lag}\eqref{it:LagEq} imply a stronger statement. Namely, under the conditions of Theorem \ref{thm:X} or of Theorem \ref{thm:AltX}, $\bf G$ acts on the preimage $\mu_{\bf X}^{-1}(\cO(\pi))$ with an open orbit. We do not know whether this holds for general $\bf X$.
\end{remark}

\subsection{\FreSp s, their duals and tensor products}
All the topological vector spaces considered in this paper will be either nuclear \Fre{\,}  or dual nuclear \Fre. For a nuclear \FreSp   $\,V$, $V^*$ will denote the strong dual, and for a \Fre$\,$ or dual \FreSp $\,W$, $V\hot W$ will denote the completed projective tensor product and $L_b(V,W)$ will denote the space of bounded linear operators from $V$ to $W$ (see \cite[\S 32]{Tre}). The projective topology on $V\hot W$ is generated by seminorms which are largest cross-norms corresponding to pairs of generating semi-norms on $V$ and $W$, see \cite[\S 43]{Tre}. In particular, if $V$ and $W$ are nuclear \FreSp s, then so is $V\hot W$. The projective tensor product of nuclear spaces is naturally isomorphic to the injective one, see \cite[Theorem 50.1]{Tre}. We will need the next proposition, which follows from \cite[Proposition 50.5 and (50.19)]{Tre}.
\begin{prop}\label{prop:TensorDual}
Let $V$ and $W$ be nuclear \FreSp s. Then we have natural isomorphisms
$$(V\hot W)^*\cong V^*\hot W^*\cong L_b(V,W^{*}).$$
\end{prop}

We will also use the following lemma.
\begin{lem}[{\cite[Lemma A.3]{CHM}}]\label{lem:flat}
Let $W$ be a nuclear \FreSpr and let
$$0\to V_1\to V_2\to V_3\to 0$$
be a short exact sequence of  nuclear \FreSp s. Then the sequence
$$0\to V_1\hot W\to V_2 \hot W \to V_3\hot W\to 0$$
is also a short exact sequence.
\end{lem}

\begin{prop}[{\cite[Lemma A.4]{CHM}}]\label{prop:TenLim}
Let $V=\lim\limits _\to V_i$ be a Hausdorff inductive limit of \Dima{a sequence  of} dual nuclear  \FreSp s, and $W$ be a dual nuclear \FreSp . Then we have a natural isomorphism $$V\hot W  \cong \lim_\to (V_i\hot W).$$
\end{prop}

\subsection{Schwartz functions and tempered distributions on real algebraic manifolds}\label{subsec:Sc}
Let ${\bf X}$ be an  algebraic manifold ({\it i.e.} smooth algebraic variety) defined over $\R$ and $X:={\bf X}(\R)$. If $\bf X$ is affine then the \FreSpr  $\Sc(X)$ of
Schwartz functions on $X$ consists of smooth complex valued functions that decay, together with all their derivatives,
faster than any polynomial. This is a nuclear \FreSp, with the topology given by the system of semi-norms $|\phi|_{d}:=\max_{x\in X}|df|$, where $d$ runs through  all differential operators on $X$ with polynomial coefficients.

For a Zariski open affine subset ${\bf U}\subset {\bf X}$, the extension by zero of a Schwartz function on $U={\bf U}(\R)$ is a Schwartz function on $X$. This enables to define the Schwartz space on any algebraic manifold ${\bf X}$, as the sum of the Schwartz spaces of the open affine pieces, extended by zero to functions on $X$.
For the precise definition of this notion see e.g. \cite{AGSc}. Elements of the dual space $\Sc^*(X)$ are called tempered distributions. The spaces $\Sc^*(U)$ for all Zariski open ${\bf U}\subset {\bf X}$ form a sheaf.

In a similar way one can define the space $\Sc(X,V)$ of $V$-valued Schwartz functions for any  \FreSpr $V$.
Namely, for an affine $X$ we demand that $q(d\phi(x))$ is  bounded for any differential operator  $d$ on $X$ and any seminorm $q$ on $V$.
It is easy to see that 
$\Sc(X,V)\cong \Sc(X)\hot V$,  (cf. \cite[Theorem 51.6]{Tre}).
We  define the  tempered distributions $\Sc^*(X,V)$ to be the continuous linear dual space.
Note that by Proposition \ref{prop:TensorDual} we have
\begin{equation}\label{=dist}
\Sc^*(X,V)\cong \Sc^*(X)\hot V^*\cong L_b(\Sc(X),V^*)
\end{equation}

If a group $G$ acts on $X$ and on $V$ then we consider the diagonal action on $\Sc(X,V)$ and the dual action on $\Sc^*(X,V)$. We denote the space of invariants in $\Sc^*(X,V)$ by $\Sc^*(X,V)^G$. We call the elements of this space \emph{equivariant distributions}.

For an algebraic bundle $\cE$ over $X$, Schwartz sections are defined as finite sums of extensions by zero from open sets on which the bundle trivializes. See \cite[\S 5]{AGSc} for more details. For a nuclear \FreSpr $V$ we denote  $\Sc(X,\cE,V):=\Sc(X,\cE)\hot V$ and $\Sc^*(X,\cE,V):=\Sc(X,\cE,V)^*$.

We will use the following version of the Schwartz kernel theorem.
\begin{prop}[{\cite[Corollary 2.6.3]{AG_RhamShap}}]\label{prop:prod}
For any real algebraic manifold $Y$, and algebraic bundles $\cE$ over $X$ and $\cE'$ over $Y$, we have
$$\Sc(X \times Y,\cE\boxtimes\cE')\cong\Sc(X,\cE) \ctp \Sc(Y,\cE').$$
\end{prop}

Let ${\bf U}\subset {\bf X}$ be a Zariski open subset, write $U:={\bf U}\cap X$ and let $Z$ denote the complement to $U$ in $X$.

%

\begin{thm}[{\cite[ Theorem 5.4.3]{AGSc}}] \label{pOpenSet}
We have
$$\Sc(U,\cE) \cong \{\phi \in \Sc(X,E)| \quad \phi \text{ is 0 on } Z \text{ with all derivatives} \}.$$
In particular, extension by zero defines a closed imbedding $\Sc(U,\cE) \into \Sc(X,\cE)$.
\end{thm}

\begin{cor}
The restriction map $\Sc^*(X,\cE)\to \Sc^*(U,\cE)$ is onto.
\end{cor}
Denote the kernel of this map by $\Sc^*_Z(X,\cE)$, and the kernel of the corresponding map $\Sc^*(X,\cE,V)\onto \Sc^*(U,\cE,V)$ by $\Sc^*_Z(X,\cE,V)$.

\begin{notn}
Denote by $F_n(X,Z,\cE)$ the space of Schwartz functions on $X$ that vanish on $Z$ with first $n$ derivatives. Denote by $F^n_Z(X,\cE)$ the orthogonal complement to $F_n(X,Z,\cE)$ in $\Sc^*_Z(X,\cE)$.
If $Z$ is smooth, denote by $N_Z^X$ the normal bundle to $Z$ in $X$ and by $\Sym^n(N_Z^X)$ the $n$-th symmetric power of this bundle.
\end{notn}

\begin{thm}[{\cite[Corollary 5.5.4]{AGSc}}]\label{thm:Filt}$\,$
\begin{enumerate}[(i)]
\item For every $n$, there is a natural isomorphism $$F^{n}_Z(X,\cE)/F^{n-1}_Z(X,\cE)\cong \Sc^*(Z,\cE\otimes\Sym^n(CN_Z^X))\,\,.$$
\item $\Sc^*_Z(X,\cE)=\bigcup_{n\geq 0}(F^n_Z(X,\cE))$.
\end{enumerate}
\end{thm}

Using Proposition \ref{prop:TenLim} we obtain the following corollary.

\begin{cor}\label{cor:Filt}
$\Sc^*_Z(X,\cE,V)=\bigcup_{n\geq 0} F^n_Z(X,\cE)\hot V^*$
\end{cor}

\begin{theorem}[{\cite[Lemma 2.8]{GGS2}}]\label{thm:Frob}
Suppose that the action of $G$ on $X$ is transitive,  fix $x\in X$ and let $G_x$ be its stabilizer in $G$. Let $\Delta_{G_x}$ denote the modular function of $G_x$. Then for any $\pi\in \cM_{\infty}(G)$ we have
$\H_0(G,\Sc(X,\cE)\hot \pi)\cong \H_0(G_x,\cE_x\otimes \pi\otimes \Delta_{G_x})$.
\end{theorem}

\subsection{Proof of Theorem \ref{thm:X}} \label{subsec:PfX}


The proof is by induction on the number $k$ of $\bfG$ orbits in $\bf X$. The base case is $k=1$, and in this case $X$ is a disconnected union of finitely many orbits of $G$. Thus the base follows from Theorem \ref{thm:Frob} and Corollary \ref{cor:CW}.

For the induction step, let $\bf Z$ be a closed $\bf G$-orbit in $\bf X$, and let $Z:= {\bf Z} \cap X$. Let $U:=X\setminus Z$.
Let $\tau:=\widetilde{\pi}$. If we have a map $\Sc(X,\cE)\onto \pi$ then
$\Sc^*(X,\cE,\tau)^G\neq 0$. Consider the exact sequence
$$0\to \Sc^*_Z(X,\cE,\tau)^G \to \Sc^*(X,\cE,\tau)^G\to \Sc^*(U,\cE,\tau)^G \, \,.$$
If $\Sc^*(U,\cE,\tau)^G\neq 0$  then the theorem follows from the induction hypothesis. Otherwise, we have $\Sc^*_Z(X,\cE,\tau)^G\neq 0$. By Theorem \ref{thm:Filt} and Corollary \ref{cor:Filt}, $\Sc^*_Z(X,\cE,\tau)$ has an exhaustive $G$-invariant filtration $F^n_Z(X,\cE)\hot \tau$ with successive quotients isomorphic to $\Sc^*(Z,\cE\otimes\Sym^n(CN_Z^X),\tau)$. Thus $$\Sc^*(Z,\cE\otimes\Sym^n(CN_Z^X),\tau)^G\neq 0$$
for some $n\geq 0$. The theorem follows now from the base of the induction.
\proofend

\subsection{Proof of Theorem \ref{thm:AltX}} \label{subsec:PfAltX}
For the proof we will need the following theorem.

\begin{thm}\label{thm:fin}
The space $\H_0(\fg,\Sc(X,\cE)\hot \pi)$ is Hausdorff and finite-dimensional.
\end{thm}
\begin{proof}
By the Casselman embedding theorem (cf. \cite[Theorem 8.21]{CaM}),
$\pi$ is a quotient of the nuclear \FreSpr$\H_0(\fp,\Sc(G,\cE')\otimes \chi)$, where $\fp$ is the Lie algebra of a minimal parabolic subgroup $\bf P \subset G$ is a minimal defined over $\R$, $\chi$ is a character of $\fp$ that is trivial on its nilpotent radical, and $\cE'$ is an algebraic bundle on $G$.
By Proposition \ref{prop:prod} we have a natural isomorphism
$$H_0(\fg,\Sc(X,\cE)\hot\H_0(\fp,\Sc(G,\cE')\otimes \chi))\cong \H_0(\fg\times \fp,\Sc(X\times G,\cE\boxtimes \cE')\otimes \chi),$$
where $G$ acts on itself by left multiplications, and on $X\times G$ diagonally, and $P$ acts on $G$\ by right multiplications, thus commuting with the action of $G$.
 Since $\bf X$ is spherical, $\bf P$ has finitely many orbits on $\bf X$, and thus $\bf G \times P$ has finitely many orbits on $\bf X\times G$. Thus, \cite[Theorem C]{AGKL} implies that $\H_0(\fg\times \fp,\Sc(X\times G,\cE\times \cE')\otimes \chi)$ is Hausdorff and finite-dimensional.
Therefore, so is $\H_0(\fg,\Sc(X,\cE)\hot \pi)$.
\end{proof}


\begin{proof}[Proof of Theorem \ref{thm:AltX}]
Let $\fin := \H_0(\fg,\Sc(X,\cE)\hot \pi)$. It is a representation of $G$, on which the connected component of the unit element acts trivially.
By Theorem \ref{thm:fin}, $\fin$ is Hausdorff and finite-dimensional.
The projection map $\Sc(X,\cE)\hot \pi\onto \sigma$ defines a non-zero map
$\xi:\Sc(X,\cE)\otimes \sigma^*\to \pi^*$. Since $\Sc(X,\cE)\otimes \sigma$ is a smooth representation, and $\pi$ is irreducible, the image of $\xi$ is the contragredient representation $\widetilde{\pi}$. Since $\Sc(X,\cE)\otimes \sigma\cong \Sc(X,\cE\otimes \sigma)$, and $\cO(\pi)=\cO(\widetilde{\pi})$, the theorem follows from Theorem \ref{thm:X}.
\end{proof}

\DimaC{
\section{Beyond spherical subgroups}\label{sec:beyond}
In this section we let $\fr\subset \fg$ be any, not necessary spherical, subalgebra.

We start with the following generalization of \Cref{thm:main}.


\begin{theorem}\label{thm:Gen}
Let $\sigma\in \cM_{f.d.}(\fr)$ let $V\in \Irr(\fg)_{\fr,\sigma}$. Suppose that for any $\bfG$-orbit $\cO\subset \AnV (V)$ we have
\begin{equation}\label{=GenCon}
2\dim \cO\cap \fr^{\bot} \leq \dim \cO
\end{equation}
 Then
\begin{equation}
2\dim \cO(V)\cap \fr^{\bot} = \dim \cO(V)
\end{equation}
\end{theorem}
The proof of \Cref{thm:main} adapts verbatim to \Cref{thm:Gen} as follows. 
\begin{proof}[Proof of \Cref{thm:Gen}]
By \Cref{prop:GenMain} we have $2\dim \overline{\cO(V)}\cap \fr^{\bot} \geq \dim \cO(V)$. By \eqref{=GenCon} we have
$$2\dim \cO(V)\cap \fr^{\bot} \leq \dim \cO(V) \text{ and }2\dim (\overline{\cO(V)}\setminus \cO(V))\cap \fr^{\bot} < \dim \cO(V).$$ Thus $2\dim \cO(V)\cap \fr^{\bot} = \dim \cO(V)$.
\end{proof}
\DimaE{As in Remark \ref{rem:Irr}, one can ease the assumption that $V$ is irreducible, and require only that there exists an $\fr$-equivariant map $\xi:V\to \fin$ that   does not vanish on any non-zero $\fg$-submodule of $V$.   }

Similarly one obtains a generalization of \Cref{thm:maintwist}. To formulate it we will need some notation. Let $h\in \fg$ be a semi-simple element such that its adjoint action $\ad(h)$ has integer eigenvalues and preserves $\fr$. Let $\chi\in \fg^*$ such that  $\ad^*(h)(\chi)=2\chi$.  For any $i\in \bZ$ let $\fr(i)$ denote the $i$-th eigenspace of $ad(h)$ on $\fr$, and let $\fn:=\bigoplus_{i< 0}\fr(i)$. \DimaF{Suppose also that for every $i<-2$ we have $\mathfrak{r}(i)=\mathfrak{g}(i)$.}

\begin{thm}\label{thm:GenGenTwist}
 Let $\sigma\in \cM_{f.d.}(\fr)$ and suppose that $\sigma(X)=\chi(X)\mathrm{Id}$ for any $X\in \fn$.
 \DimaE{Let $V$ be a $\fg$-module. Suppose that there exists an $\fr$-equivariant map $\xi:V\to \fin$ that does not vanish on any non-zero $\fg$-submodule of $V$. Suppose also that} for any $\bfG$-orbit $\cO\subset \AnV (V)$ we have
\begin{equation}\label{=GenTCon}
2\dim \cO\cap (\fr^{\bot}+\chi) \leq \dim \cO
\end{equation}
 Then \DimaE{for some orbit $\cO_{\max}\subset \AnV(V)$ of dimension $\dim \AnV(V)$ we have}
\begin{equation}\label{=GenT}
2\dim \cO_{\max}\cap (\fr^{\bot}+\chi) = \dim \cO_{\max}
\end{equation}
\end{thm}

For the proof we will need the following generalization of \Cref{lem:TwiAs}.
\begin{lemma}\label{lem:GenTwiAs}
Let $V$ be a $\fg$-module generated by a finite-dimensional $\fr$-invariant vector subspace $W$ on which $\fn$ acts via $\chi|_{\fn}$. Then $\Asv^{\Kaz}(V)$ \DimaF{is a non-empty subset of} $\chi+\fr^{\bot}$.
\end{lemma}
\begin{proof}
Define $V_k:=F_k^{\Kaz}(\cU)W$ for any $k\in \bZ$. 
\DimaF{By Lemma \ref{lem:non0}, $V_{-1}=0$ and thus $\Asv^{\Kaz}(V)$ is not empty.}
Since $\ad(h)$ preserves $\fr$, we have $\fr=\bigoplus_{i}\fr(i)$.
For any $X\in \fr(i)$ with $i<0$ we have
$$(X-\chi(X))V_k=(X-\chi(X))F_k^{\Kaz}(\cU)W\subset [X,F_k^{\Kaz}(\cU)]W\subset F_{k+i}^{\Kaz}(\cU)W=V_{k+i}$$
For any $X\in \fr(i)$ with $i\geq0$ we have
$$(X-\chi(X))V_k=XV_k= XF_k^{\Kaz}(\cU)W\subset [X,F_k^{\Kaz}(\cU)]W+F_k^{\Kaz}(\cU)W\subset F_{k+i}^{\Kaz}(\cU)W=V_{k+i}$$
Since $X\in F_{i+2}^{\Kaz}(\cU)$, and $\chi(X)=0$ unless $i=-2$, we obtain that $X-\chi(X)$ acts by zero on $\gr^{\Kaz}(V)$. Thus $\Asv^{\Kaz}(V)\subset \chi+\fr^{\bot}$.
\end{proof}

\begin{proof}[Proof of \Cref{thm:GenGenTwist}]
Let $M\subset \Hom_{\C}(V,\sigma^*)$ be the $\fg$-module generated by $\xi$. \DimaE{Then we have $\Ann(M)=\Ann(V)$. Indeed, the inclusion $\Ann(V)\subset \Ann(M)$ is obvious, and for the other direction we note that $\Ann(M)V\subset V$ is a submodule on which $\xi$ vanishes, and thus $\Ann(M)V=0$ by our condition on $\xi$. Thus $\Anv(M)=\Anv(V)$.
By Corollary \ref{cor:Kaz} and Theorem \ref{thm:TwiGab}, we have
\begin{equation}\label{=GenTBer}
\dim \Anv(M)\leq 2\dim \Asv^{\Kaz}(M)
\end{equation}
On the other hand, by Lemma \ref{lem:GenTwiAs} we have
\begin{equation}\label{=GenTInt}
\Asv^{\Kaz}(M)\subset \Anv(M) \cap (\chi+\fr^{\bot})=\Anv(V)\cap (\chi+\fr^{\bot}).
\end{equation}
By the condition of the theorem, for every nilpotent orbit $\cO\subset \AnV(V)$ we have
\begin{equation}\label{=GenTIntDim}
2\dim (\cO\cap (\chi+\fr^{\bot}))\leq \dim \cO.
\end{equation}

From (\ref{=GenTBer},\ref{=GenTInt}, \ref{=GenTIntDim}) we obtain \eqref{=GenT}.}
\end{proof}

From Theorems \ref{thm:Gen} and \ref{thm:GenGenTwist} we obtain the following generalization of  \Cref{cor:CW}.

\begin{cor}\label{cor:genCW}
Let $\fin\in \cM_{f.d.}(\fr)$  and  let $\pi\in \Irr(G)$. Assume that  there exists a  continuous non-zero $\fr$-equivariant linear map $\pi \to \fin$. Then

\begin{enumerate}[(i)]
\item If
for any $\bfG$-orbit $\cO\subset \AnV (\pi)$ we have
\begin{equation}\label{=GenCon2}
2\dim \cO\cap \fr^{\bot} \leq \dim \cO
\end{equation}
Then  $2\dim\cO(\pi)\cap \fr^\bot=\dim\cO(\pi)$.

\item Assume further that there exists a semi-simple element $h\in \fg$ such that $\ad(h)$ preserves $\fr$ and has integer eigenvalues, and a functional $\chi\in \fg^*(2)$  such that
$\fn:=\bigoplus_{i< 0}\fr(i)$ acts on $\fin$ via $\chi|_\fn$.
If
for any $\bfG$-orbit $\cO\subset \AnV (\pi)$ we have
\begin{equation}\label{=GenCon3}
2\dim \cO\cap (\chi+\fr^{\bot}) \leq \dim \cO
\end{equation}

Then  $$2\dim \left(\cO(\pi)\cap (\chi+\fr^{\bot})\right)=\dim \cO(\pi).$$
 \end{enumerate}
\end{cor}
The deduction of this corollary follows the proof of \Cref{cor:CW} verbatim.

In the following two subsections we apply parts (i) and (ii)\ of this corollary respectively.

\subsection{Howe correspondence in type II}

Let $V$ and $W$ be real vector spaces, and denote $U:= \Hom_{\R}(V,W)$.
 The natural actions of $\GL(V)$ and $\GL(W)$ on $U$\ and $U^*$ provide an embedding
$$\iota:\GL(V)\times \GL(W)\into \Sp(U\oplus U^*)$$
Let $\widetilde{Sp}(V,W)$ denote the metaplectic double cover of $\Sp(U\oplus U^*)$.  Since $\GL(V)\times \GL(W)$ preserves the Lagrangian subspace $U\subset U\oplus U^*$, this double cover splits over the image of $\iota$ and thus we have a natural embedding of $\GL(V)\times \GL(W)$ into $\widetilde{Sp}(V,W)$, that we will also denote by $\iota$ by abuse of notation.
Let $\varpi$ denote the Weil representation of $\widetilde{Sp}(V,W)$.

Let $\mu$ denote the moment map $\mu:T^*U\cong U\oplus U^*\to \gl(V)^*\times \gl(W)^*$ corresponding to the action of $\GL(V)\times \GL(W)$ on $U$.
To obtain an explicit formula for $\mu$ one can identify $\gl(V)\cong \gl(V)^*$ and $\gl(W)\cong \gl(W)^*$\ using the trace form, and further $U^*\cong \Hom_{\R}(W,V)$. Then we have
\begin{equation}
\mu(X,Y):= (YX,XY)
\end{equation}

We will now derive from Corollary \ref{cor:genCW}  the following theorem.

\begin{thm}\label{thm:theta}
Let $\pi \in \Irr(\GL(V))$ and $\tau\in \Irr(\GL(W))$\ such that $\pi\boxtimes \tau$ is an irreducible quotient of $\varpi|_{\GL(V)\times \GL(W)}$.
Then $\cO(\pi)\times \cO(\tau)$ lies in the image of the moment map $\mu$.
\end{thm}

In order to apply Corollary \ref{cor:genCW} let $G:=\GL(V)\times \GL(W)\times \widetilde{Sp}(V,W)$ and let $R=\mathrm{Graph}(\iota)\subset G$.
Note that $\cO(\varpi)$ is the minimal nilpotent orbit $\cO_{\min}$, and $\dim \cO_{\min}=2\dim V \cdot \dim W$.

\begin{prop}[I. Karshon, see Appendix \ref{app:Ido}]\label{prop:Ido}
Let $\cO_1\subset \gl(V)$ and $\cO_2\subset \gl(W)$ be nilpotent orbits. Then
\begin{enumerate}[(i)]
\item \label{it:Ido1}$\cO_1\times \cO_2\subset \Im \mu$ if and only if $\cO_1\times \cO_2\times \cO_{\min}$ intersects $\fr^{\bot}$
\item \label{it:Ido2}If $\cO_1\times \cO_2\times \cO_{\min}$ intersects $\fr^{\bot}$ then
\begin{equation}
2 \dim (\cO_1\times \cO_2\times \cO_{\min})\cap \fr^{\bot}=\dim \cO_1\times \cO_2\times \cO_{\min}
\end{equation}
\end{enumerate}
\end{prop}
Theorem \ref{thm:theta} follows now from \Cref{cor:genCW} and \Cref{prop:Ido}.
Related results were obtained in \cite{Prz,LokeMa,GZ,Prz2}.

\DimaE{We will say that $\cO_1$ and $\cO_2$ \emph{match} if
$\cO_1\times \cO_2\subset \Im \mu$.  This condition can be made explicit in terms of partitions in the following way.}
\begin{prop}[Appendix \ref{app:Ido}]\label{prop:match}
Orbits $\cO_1$ and $\cO_2$ match if and only if the corresponding partitions $\lambda(\cO_1)$ and $\lambda(\cO_2)$ satisfy
$|\lambda(\cO_1)_i-\lambda(\cO_2)_i|\leq 1$ for any index $i$.
\end{prop}

Here, we set $\lambda(\cO)_j=0$ if $j>length(\lambda(\cO))$.
This proposition implies that
if $\dim W\geq \dim V$ then the \DimaE{maximal among all orbits  that match $\cO_1$} is given by
\begin{equation}\label{=TO}
\lambda(\cO_2)_i=\begin{cases}
\lambda(\cO_1)_i+1 & 1\leq i \leq \dim W-\dim V \\
\lambda(\cO_1)_i & i>\dim W-\dim V \\
\end{cases},
\end{equation}
Furthermore, if $\dim W \geq 2\dim V$ or if the matrix rank on $\cO_1$ is at least $2\dim V - \dim W$ then
\eqref{=TO} defines the only matching orbit with matrix rank equal to $\dim V$. If $\rk(\cO_1)<2\dim V - \dim W$ then \Cref{prop:match} implies that $\cO_1$ has no matching orbit of rank $\dim V$.
\begin{rem}
Similar results with the same proof hold if $V$ and $W$ are complex vector spaces.
\end{rem}

\subsection{Degenerate Whittaker models}

In this subsection we deduce from Corollary \ref{cor:genCW}(ii) the main result of \cite{Mat}. Let $h\in \fg$ be a semi-simple element. Suppose that $\ad(h)$ has integer eigenvalues, and let $\fr:=\bigoplus_{i\leq -2}\fg(i)$. Let $\chi\in \fg^*(2)$. Note that it defines a character of $\fr$, that we will also denote by $\chi$ by abuse of notation.

\begin{thm}[{\cite[Theorem 2]{Mat}}]\label{thm:Mat}
\DimaE{Let $V$ be a $\fg$-module that has a non-zero $({\fr,\chi})$-equivariant functional.} Then $\chi \in \AnV(V)$.
\end{thm}

For the proof we will need the following lemma.

\begin{lemma}\label{lem:Mat}
Any nilpotent orbit $\cO$ that intersects $\chi +\fr^{\bot}$ includes $\chi$ in its closure $\overline{\cO}$.
\end{lemma}

\begin{proof}
Choose $x\in \cO\cap (\chi +\fr^{\bot})$. For any $t\in \bC$ let $x_t:=t^{-2}\exp(th)x$. Since $\cO$ is conic and $\bfG$-invariant, we have $x_t\in \cO$. Since $\chi\in \fg^*(2)$ and $\fr^{\bot}=\bigoplus_{i< 2}\fg^*(i)$, we have $\lim_{t\to 0}x_t=\chi$. Thus $\chi \in \overline{\cO}$.
\end{proof}

\begin{proof}[Proof of \Cref{thm:Mat}]
\DimaE{Let $\xi\neq 0$ be a non-zero  $(\fr,\chi)$-equivariant functional on $V$. Let $V'$ be the quotient of $V$ by the maximal submodule on which $\xi$ vanishes.

Assume by way of contradiction that  $\chi \notin \AnV(V')$. By \Cref{lem:Mat}  this implies that for any orbit $\cO'\subset \AnV(V')$, the intersection $\cO'\cap ( \chi +\fr^{\bot})$ is empty. Thus \Cref{thm:GenGenTwist} implies that some orbit $\cO\subset \AnV(V')$  intersects $\chi +\fr^{\bot}$. Applying \Cref{lem:Mat} again we obtain $\chi\in \overline{\cO}\subset \Anv(V')\subset \AnV(V)$.}
\end{proof}

}
\section{Conjectures}\label{sec:conj}
In this section let $F$ be a non-Archimedean local field. Let $\bf G$ be an algebraic reductive group defined over $F$ and let $G$ be the group of its $F$-points. Then $G$ is an locally compact totally disconnected group, and
we denote by $\Irr(G)$ the collection of smooth irreducible representations of $G$ in complex vector spaces. Smooth here means that every vector has an open stabilizer.

For any  $\pi\in\Irr(G)$, Harish-Chandra defined a complete invariant - the character  $\chi_{\pi}$ of $\pi$. It is a generalized function on $G$. We denote by $\overline{\WF(\pi)}$ the closure of the wave-front set of $\chi_{\pi}$ at the unit of $G$. This is a subset of $\fg^*(F)$. We refer the reader to \cite[\S 8]{Hor} and \cite[\S 2.8.6]{CHLR} for the definition of wave-front set of a generalized function. Let us now give an equivalent description for $p$-adic $F$.

If $F$ has characteristic zero then $\chi_{\pi}$ defines a generalized function $\zeta_{\pi}$ on a neighborhood of zero in the Lie algebra ${\mathfrak{g}(F)}$ of $G$,
by restriction to a neighborhood of $1\in G$ and applying logarithm.
By \cite{HowGL},\cite[p. 180]{HCWF}, $\zeta_{\pi}$ is  a linear combination of Fourier transforms of $G$-measures of nilpotent coadjoint orbits.
The measures  extend to $\fg^*(F)$ by \cite{RangaRao}.
Then $\overline{\WF(\pi)}$ is the union of closures of nilpotent orbits $\cO$ such that the Fourier transform of the invariant measure on $\cO$ enters the decomposition of $\zeta_{\pi}$ with a non-zero coefficient.

\begin{notn}\label{notn:WF}
Define the wave front set $\WF^{\max}(\pi)$ to be the union of all nilpotent $G$-orbits $\cO\subset \overline{\WF(\pi)}$, such that for any nilpotent orbit $\cO'$ with $\cO\subset \overline{\cO'}\subset \overline{\WF(\pi)}$ we have $\cO=\cO'$.
We consider $\WF^{\max}(\pi)$ as a subset of $\fg^*$ by embedding $\fg^*(F)$ into $\fg^*$.
\end{notn}

Let $\bf H\subset G$ be a spherical subgroup defined over $F$, and let $H$ denote its group of $F$-points. Denote by $\Irr(G)_H$ the collection of all irreducible representations of $G$ possessing non-zero $H$-invariant linear functionals.
For a  character $\phi$ of $H$ we will use the analogous notation $\Irr(G)_{(H,\phi)}$.
Let $\bfG\cdot\fh^{\bot}$ denote the image of the coadjoint action map $\bfG\times \fh^{\bot}\to \fg^*$. Note that it coincides with the image of the moment map $T^*(\bfG/\bfH) \to \fg^*$.

\begin{conj}\label{conj:main}
For any $\pi\in \Irr(G)_{H}$, we have $\WF^{\max}(\pi)\subset \bfG \cdot\fh^{\bot}$.
\end{conj}

To generalize this conjecture, let $\phi$ be a  character of $H$. Assume that there exists a decomposition $H=S\ltimes N$ such that $N$ is a parabolic nilradical. Let $\fs$ and $\fn$ denote the Lie algebras of $S$ and $N$.
Let $\chi\in \fg^*$ such that $\chi|_{\fs}=0$, and $\chi|_{\fn}$ determines $\phi|_N$.

\begin{conj}\label{conj:mainTwist}
Under the assumptions above, for every $\pi\in \Irr(G)_{(H,\phi)}$ we have $$\WF^{\max}(\pi)\subset \bfG \cdot(\chi+\fh^{\bot}).$$

Moreover, for every coadjoint $\bfG$-orbit $\cO$ that intersects $\WF^{\max}(\pi)$ we have
$$\dim \cO\cap (\chi+\fh^{\bot})=\dim \cO/2.$$
\end{conj}

This conjecture would allow to extend to the non-Archimedean case the results described in \S\S \ref{sec:appl}-\ref{sec:TwistAppl}. In particular, the non-Archimedean analogue of Corollary \ref{cor:strong} applied to the case $H=G$ (which corresponds to the diagonal  pair $\Delta G \subset G\times G$) implies the following conjecture.

\begin{conj}[\cite{MW}]\label{conj:irr}
For any $\pi\in \Irr(G)$, $\WF^{\max}(\pi)$ lies in a single $\bfG$-orbit.
\end{conj}

The following version of Conjecture \ref{conj:main} is proven in \cite{GS_X}.

\begin{thm}[{\cite[Corollary C]{GS_X}}]\label{thm:weak}
Suppose that $F$ has characteristic zero. Let $\phi$ be a character of $H$ that is
trivial on its unipotent radical, and let $\pi\in \Irr(G)_{(H,\phi)}$. Then
$$\overline{\WF(\pi)}\subset \overline{G\cdot\fh^{\bot}(F)}.$$
\end{thm}
Here, $\fh^{\bot}(F)$ denotes the rational $F$-points of $\fh^{\bot}$, and $\overline{G\cdot\fh^{\bot}(F)}$ denotes the closure in the (Hausdorff) local topology on $\fg^*(F)$.

\begin{remark}
\cite[Corollary C]{GS_X} does not require $H$ to be spherical. It also  works in the generality of non-homogeneous $G$-spaces.
\end{remark}

This theorem, together with Corollary \ref{cor:CW} is our main general evidence for Conjectures \ref{conj:main} and \ref{conj:mainTwist}.

In addition, several special cases are known to hold true for $\bf G=\GL_n$ and $p$-adic $F$.

First of all, Conjecture \ref{conj:irr} is proven for $\GL_n(F)$ in \cite[\S II.2]{MW}, which also expresses the corresponding partition, that we will denote $\lambda(\pi)$, through the Zelevinsky classification \cite{Zel}.
The Arthur type can also be expressed through the Zelevinsky classification (see {\it e.g.} \cite{OSsl2}).

Thus, the conditions on Arthur parameters in \cite[Conjecture 5.1]{GGP} imply that $|\lambda_i^t(\pi)-\lambda^t_i(\tau)|\leq 1$, as in Theorem \ref{thm:strongPart}.
The non-Archimedean case of this conjecture is proven in \cite{Chan}
(a partial result was obtained in \cite{MGur}).

\begin{remark}
In fact, the conditions in \cite[Conjecture 5.1]{GGP} further  imply that for any $i$ with $\lambda(\pi)^t_i=\lambda(\tau)^t_i$ we also have $\lambda(\pi)^t_i=\lambda(\tau)^t_i=1$. However, this additional property does not always hold for general smooth representations. Indeed, let $n:=4, \, k:=1$, $\pi:=1_3\times 1_2$ be the Bernstein-Zelevinsky product of trivial representations of $\GL_3(F)$ and $\GL_2(F)$ and $\tau:=|\det|_2^{1/2}\times 1_2\in \Irr(\GL_4(F))$. Then $\Hom_{H}(\pi|_{H},\tau)\neq 0$, but $\lambda^t(\pi)=(3,2)$ and $\lambda^t(\tau)=(2,2)$.
This does not contradict \cite{Chan}, since $\tau$ is not of Arthur type.
\end{remark}

The $p$-adic analogue of Corollary \ref{cor:Klya} is proven in
\cite{OSKlyaZel}.

\DimaE{The results of \cite{GZ}, that provide a certain analogue of \Cref{thm:theta} for type I dual pairs, are proven uniformly for all local fields of characteristic zero. Finally, the non-Archimedean analogue of Theorem \ref{thm:Mat} is proven in a much stronger form in \cite{MW}.}
\DimaC{
\appendix
\section{Proof of Propositions \ref{prop:Ido} and \ref{prop:match}\\ by Ido Karshon}\label{app:Ido}


\subsection{Proof of \Cref{prop:Ido}}
Identify  $U\oplus U^*$ with $Hom(V,W)\oplus Hom(W,V)$,
with the symplectic form given by
\[
\omega((X,Y),(X',Y'))=tr(YX')-tr(XY')
\]
Then the differential $d\iota: \gl(V)\times\gl(W)\to\mathfrak{sp}(U\oplus U^*)$ of $\iota$ at $(\mathrm{Id},\mathrm{Id})$ is
given by $(A,B)\mapsto(R_{A}+L_{B},-R_{B}-L_{A})$ where $R_A$ and $R_B$ denote the
operators of right multiplication by $A$ and $B$ respectively, and $L_A$ and $L_B$ denote the operators of left multiplication by $A$ and $B$ respectively. Recall that $\mathfrak{r}$ is  the graph of $d\iota$.
\begin{proof}[Proof of \Cref{prop:Ido}\eqref{it:Ido1}]
 Any
$T\in\mathfrak{sp}(U \oplus U^*)$  of rank $1$ has the form $x\otimes x^*$
for $x\in U\oplus U^*$, and $x^*\in (U\oplus U^*)^*$ given by symplectic pairing with $x$. Thus $\mathcal{O}_{\min}$ consists of elements of the form $(X,Y)\otimes(-Y,X)$.
The pairing of such element with $(R_{A}+L_{B},-R_{B}-L_{A})$ is
\begin{multline*}
tr\left((X,Y)\otimes(-Y,X)\circ(R_{A}+L_{B},-R_{B}-L_{A})\right)=(-Y,X)\left((R_{A}+L_{B},-R_{B}-L_{A})(X,Y)\right)=\\=
 (-Y,X)(XA+BX,-YB-AY)=-tr(YXA)-tr(BXY)-tr(YXA)-tr(BXY)=\\
 -2(tr(YXA) + tr(XYB))
\end{multline*}
Thus $(M,N,(X,Y)\otimes(-Y,X))\in \fr^{\bot}$
 if and only if $M=2YX,N=2XY$, {\it i.e.} $(M,N)=\mu(2X,Y)$.
\end{proof}

This argument also allows to reformulate \Cref{prop:Ido}\eqref{it:Ido2} as the following lemma.

\begin{lem}\label{lem:AppKey}
 Let $\mathcal{O}_{1}\subset\mathfrak{gl}\left(V\right),\mathcal{O}_{2}\subset\mathfrak{gl}\left(W\right)$
be nilpotent orbits. Let
$$S:=\left\{ \left(X,Y\right)\in Hom\left(V,W\right)\oplus Hom\left(W,V\right):YX\in\mathcal{O}_{1},XY\in\mathcal{O}_{2}\right\} .$$
If $S$ is nonempty then
\[
\dim S=\frac{1}{2}\dim(\mathcal{O}_{1}\times\Oc_{2}\times \cO_{\min})
\]
\end{lem}

To prove this lemma we will need the following key lemma.
\begin{lem}\label{lem:ParDim}
Let $U,L$ be vector spaces. Suppose $\Oc$ is a nilpotent orbit in
$\mathfrak{gl}\left(U\right)$ and \DimaE{fix} $K\in\Oc$. Let $\Oc'$ be a nilpotent
orbit in $\mathfrak{gl}\left(U\oplus L\right)$ and let $T$ be the
collection of all linear maps $A:L\to U$ satisfying $\left(\begin{array}{cc}
K & A\\
O & O
\end{array}\right)\in\Oc'$. Let $\lambda=(\lambda_{1},\lambda_{2},...),\lambda'=(\lambda_{1}',\lambda_{2}',...)$
be the partitions corresponding to $\Oc,\Oc'$. Then:
\begin{enumerate}[(i)]
\item \label{it:nonEmpty} The variety $T$ is nonempty if and only if  $\lambda_{i}\le\lambda_{i}'\le\lambda_{i}+1$
for all $i$.
\item \label{it:dimT} In that case, we have $\dim T=\frac{\dim\Oc'-\dim\Oc}{2}$.
\end{enumerate}
\end{lem}
We postpone the proof of this lemma to \S \ref{subsec:PfParDim} below.
By transposing the matrix, this lemma implies the analogous statement for block-lower-triangular matrices.
\begin{proof}[Proof of Lemma \ref{lem:AppKey}]
Denote $m=\dim V,n=\dim W$. For any $k\le\min\left(m,n\right)$ let
$$S_{k}:=\left\{ \left(X,Y\right)\in S:\text{rk}\left(X\right)=k\right\}. $$
There is an action of $GL_{m}\times GL_{n}$ on $S_{k}$ which is
transitive on the $X$ coordinate. Thus, the subspace of
$S_{k}$ consisting of pairs with $X=\left(\begin{array}{cc}
I_{k} & O\\
O & O
\end{array}\right)$ has codimension
\[
\dim\left\{ X\in Hom\left(V,W\right):\text{rk}X=k\right\} =k\left(m+n-k\right)
\]
For $\left(X,Y\right)$ in this subspace, write $Y=\left(\begin{array}{cc}
K_{k\times k} & A_{k\times(n-k)}\\
B_{(m-k)\times k} & C_{(m-k)\times(n-k)}
\end{array}\right)$. The condition for $\left(X,Y\right)\in S_{k}$ is equivalent to
$\left(\begin{array}{cc}
K & O\\
B & O
\end{array}\right)\in\Oc_{1},\left(\begin{array}{cc}
K & A\\
O & O
\end{array}\right)\in\Oc_{2}$. This condition is independent of $C$, which lies in a space of
dimension $\left(m-k\right)\left(n-k\right)$. Let
$$Z=\left\{\left(\begin{array}{cc}
K & O\\
B & O
\end{array}\right)\in\Oc_{1},\left(\begin{array}{cc}
K & A\\
O & O
\end{array}\right)\in\Oc_{2}\right\}.$$ Since $K$ must be nilpotent, we may stratify $Z$ with respect to
the nilpotent orbit $\Oc$ of $K$. Applying Lemma A.2 once to $\Oc$
and $\Oc_{1}$, and once to $\Oc$ and $\Oc_{2}$, the dimension of
each nonempty stratum is
$$(\dim\Oc_{1}-\dim \cO)/2 + (\dim \cO_2-\dim \cO)/2 + \dim \cO=(\dim \cO_1+\dim \cO_2)/2.$$
Thus, whenever $S_{k}$ is nonempty, we have
\begin{multline*}
\dim S_{k}=k\left(m+n-k\right)+\left(m-k\right)\left(n-k\right)+\frac{\dim\Oc_{1}+\dim\Oc_{2}}{2}=\\
mn+\frac{\dim\Oc_{1}+\dim\Oc_{2}}{2}=
\frac{\dim\Oc_{1}\times\Oc_{2}\times\Oc_{\min}}{2}
\end{multline*}
Since $S=\bigcup_{k=0}^n S_k$, and all non-empty strata $S_k$ have the required dimension, the lemma follows.
\end{proof}
\subsection{Proof of Proposition \ref{prop:match}}
We identify a nilpotent orbit in $\mathfrak{gl}_{n}$ with the partition
whose values are the sizes of blocks of the Jordan form of the orbit.
Let $m:=\dim V$ and $n:=\dim W$.
\begin{proof}[Proof of \Cref{prop:match}] From the considerations in the previous subsection we see that
$\Oc_{1}$ and $\Oc_{2}$ match if and only if there exist $k\le m,n$,
a nilpotent orbit $\Oc\subset\mathfrak{gl}_{k}$,
and matrices $K\in\Oc,A,B$ so that $\left(\begin{array}{cc}
K & O\\
B & O
\end{array}\right)\in\Oc_{1},\left(\begin{array}{cc}
K & A\\
O & O
\end{array}\right)\in\Oc_{2}$. By Lemma A.2, this is equivalent to the existence of a partition
$\kappa$ so that $\kappa_{i}\le\lambda(\cO_1)_{i},\lambda(\cO_2)_i\le\kappa_{i}+1$
for all $i$. This condition is equivalent to $|\lambda(\cO_1)_{i}-\lambda(\cO_2)_i|\le1$
for all $i$.\end{proof}

\subsection{Proof of Lemma \ref{lem:ParDim}}\label{subsec:PfParDim}
Let $k=\dim U,n=\dim L$. Denote $K'=\left(\begin{array}{cc}
K & A\\
0 & 0
\end{array}\right)$.
\DimaE{

\textbf{Proof of part \eqref{it:nonEmpty}.}
Suppose first that $\lambda_{i}\le\lambda_{i}'\le\lambda_{i}+1$
for all $i$. Let $i_1,\dots ,i_n$ be all the indices for which $\lambda_{i_j}'=\lambda_{i_j}+1$ (it is possible that $\lambda_i=0$).  Introduce a basis for $U$ in which $K$ is in Jordan canonical form. Fix a basis $l_1,\dots l_n$ for $L$, and define $A:L\to U$ in these bases by letting $Al_j$ be the first vector in $i_j$-th Jordan chain (of size $\lambda_{i_j}$). If $\lambda_{i_j}=0$ we set $Al_j:=0$.
Then the partition of $K'$ will be given by the $\lambda'_i$, i.e. $A\in T$ and thus $T$ is non-empty.

To the other direction, assume that $T$ is non-empty.
Let us compare the Jordan chains of $K'$ to the Jordan chains of
$K$. Since the image of $K'$ lies in $U$, its maximal Jordan chain
has either the same length as the longest chain of $K=K'|_U$, or is longer
than it by 1 vector. Equivalently, $\lambda_{1}\le\lambda_{1}'\le\lambda_{1}+1$.

If $\lambda_{1}=\lambda_{1}'$ then any maximal
Jordan chain for $K$ is also a maximal chain for $K'$. Pick one
such chain and let $U_{1}\le U$ be the space it spans. We get nilpotent
operators $\overline{K'},\overline{K}$ on $U/U_{1}\oplus L$ and
$U/U_{1}$ that satisfy the assumptions of the theorem and correspond
to the partitions $\left(\lambda_{2},\lambda_{3},...\right),\left(\lambda_{2}',\lambda_{3}',...\right)$.
The claim that $\lambda_{i}\le\lambda_{i}'\le\lambda_{i}+1$ for all $i$  follows now by induction on $\dim U$.

In the other case, {\it i.e.} $\lambda'_{1}=\lambda_{1}+1$, there exists a maximal chain for $K'$ consisting
of a vector $v\in L$ followed by a maximal chain of $K$ (that may
be empty). Now the same argument works, if in addition to replacing
$U$ by $U/U_{1}$ we replace $L$ by $L/Span\left(v\right)$. This
induction proves that if $T$ is non-empty then $\lambda_{i}\le\lambda_{i}'\le\lambda_{i}+1$.\\

\textbf{Proof of part \eqref{it:dimT}.} In this part we assume that $\lambda_{i}\le\lambda_{i}'\le\lambda_{i}+1$
and compute the dimension of $T$.
  We will use the following lemma:
\begin{lem}[{\cite[Corollary 6.1.4]{CM}}]\label{lem:OrbDim}
Let $W$ be a vector space, and $\Oc\subset \gl(W)$ be a nilpotent orbit corresponding to a partition $\lambda=\left(\lambda_{1},\lambda_{2},...\right)$.
Let $\lambda^{t}$ be the transpose partition of $\lambda$, defined
by $\lambda_{i}^{t}=\left|\left\{ j:\lambda_{j}\ge i\right\} \right|$,
or by reflecting the Young diagram of $\lambda$ along the diagonal.
Then
\begin{equation}
\dim\Oc=\left(\dim W\right)^{2}-\sum_{i}\left(\lambda_{i}^{t}\right)^{2}.
\end{equation}
\end{lem}

We now compute the dimension of $T$ by an inductive argument similar to the one used in the proof of part \eqref{it:nonEmpty}. We consider two cases: either $\lambda_{1}'=\lambda_{1}$
or $\lambda_{1}'=\lambda_{1}+1$.

Suppose first that $\lambda_{1}'=\lambda_{1}+1$; this is the difficult
case. 
Fix a decomposition $U=U_{1}\oplus U_{2}$ where $U_{1}$ corresponds
to one maximal Jordan block of $K$ and $U_{2}$ corresponds to the
rest of the blocks. Fix a non-zero vector $v\in L$ and denote $L_1:=Span\left(v\right)$.
\begin{lem}
The condition $AL_1\nsubseteq KU+U_{2}$ defines an open dense subset $S\subset T$.
\end{lem}
\begin{proof}
The condition $AL_1\nsubseteq KU+U_{2}$ is open
and non-empty on $T$, and the fact it is dense can be seen via composing $A$ with linear automorphisms of $U$ and $L$ that are close to the identities.
\end{proof}

Let $T_2$ denote the set defined analogously to $T$, but with $U$ replaced by $U_2$, $K$ replaced by $K_2:=K|_{U_2}$, $L$ replaced by a hyperplane in $L$, and the partitions $\lambda,\lam'$ replaced by the partitions $\mu=\left(\lambda_{2},\lambda_{3},...\right),\mu'=\left(\lambda_{2}',\lambda_{3}',...\right)$. Let  $\bP^*(L)$ denote the space of hyperplanes in $L$.

\begin{lem}
There exists a dominant map $\nu:S\to U\times \bP^*(L)\times  T_2,$ with every fiber isomorphic to the space of matrices $\Mat_{n-1,\lam_1-1}(\C)$.
\end{lem}
\begin{proof}

Given $A\in S$, define $L_2\subset L$ to be the preimage $L_{2}=A^{-1}\left(KU+U_{2}\right)$. Since $AL_1\nsubseteq KU+U_{2}$ by the definition of $S$,
$L=L_1\oplus L_2$, and in particular $L_2$ is a hyperplane.
Let $A_{1}:L_{2}\to KU_{1}\oplus U_{2}$ be the restriction $A_1:=A|_{L_2}$. Denote the coordinates of $A_1$ by $B:L_{2}\to KU_{1}$ and $A_{2}:L_{2}\to U_{2}$. We define the map $\nu$ by $\nu(A):=(Av, L_2,A_2)$. Let us show that $A_2$ indeed lies in $T_2$ and that any fiber is isomorphic to the space of linear maps $\Hom_{\C}(L_2,KU_1)$.

Note that $A$ is uniquely defined by the quadruple
$\left(Av,L_{2},B,A_{2}\right)$. This  quadruple has to satisfy some condition for
$K'=\left(\begin{array}{cc}
K & A\\
O & O
\end{array}\right)$ to be in $\Oc'$, and the condition may only depend on $B,A_{2}$.


However, the condition does not depend on $B$. Indeed, pick $Av,L_{2},B,A_{2}$
and a collection of Jordan chains of $K'$. Consider one chain that
starts with $v_{2}+w$, for $v_{2}\in L_{2},w\in L_1\oplus U$.
If $B$ was changed to $\tilde B$, there would be a $u\in U_{1}$
so that $\tilde Bv_{2}-Bv_{2}=Ku$, and then replacing the first vector
$v_{2}+w$ by $v_{2}+w-u$ would preserve the rest of the chain. Doing
this to all chains at once shows that the choice of $B$ does not
affect the nilpotent orbit $K'$ is in.

Taking $B=0$ for convenience, we see that $A_2\in T_2$. Conversely, it is not difficult to see that for any  non-zero vector $w\in U$, any $L_2\in \bP^*(L)$ with $w\notin L_2$, and any $A_2\in T_2$, the map $A$ defined by the quadruple $(w,L_2,0,A_2)$ lies in $S$.
Thus, $\nu$ is dominant, and any fiber is isomorphic to $\Hom_{\C}(L_2,KU_1)$. Finally, we note that $\dim L_2=n-1$  and $\dim KU_1=\lam_1-1$.
\end{proof}

This lemma implies that
\begin{equation}\label{=T}
\dim T=k+\left(n-1\right)+\left(n-1\right)\left(\lambda_{1}-1\right)+\dim T_{2}
\end{equation}
Now note that $\mu_{i}^{t}=\max\left(\lambda_{i}^{t}-1,0\right)$.
For any  partition
$\kappa$, denote by $\Oc_{\kappa}$ the corresponding nilpotent orbit. By the induction hypothesis and Lemma \ref{lem:OrbDim} we see that}
\begin{multline}\label{=T2}
\dim T_{2}  =\frac{\dim\Oc_{\mu'}-\dim\Oc_{\mu}}{2}
=  \frac{\left(\dim L_{2}+\dim U_{2}\right)^{2}-\sum_{i}\left(\mu_{i}'^{t}\right)^{2}-\left(\dim U_{2}\right)^{2}+\sum_{i}\left(\mu_{i}^{t}\right)^{2}}{2}\\
=  \frac{\left(n-1+k-\lambda_{1}\right)^{2}-\left(k-\lambda_{1}\right)^{2}-\sum_{i}\left(\left(\mu_{i}'^{t}\right)^{2}-\left(\mu_{i}^{t}\right)^{2}\right)}{2}
\end{multline}
Further,
we have\begin{equation}\label{=sq}
  \sum_{i}\left(\left(\mu_{i}'^{t}\right)^{2}-\left(\mu_{i}^{t}\right)^{2}\right)
=  \sum_{i:\lambda_{i}'^{t}\ge1}\left(\lambda_{i}'^{t}-1\right)^{2}-\sum_{i:\lambda_{i}^{t}\ge1}\left(\lambda_{i}^{t}-1\right)^{2}
=  \sum_{i}\left(\lambda_{i}'^{t}\right)^{2}-2\left(n+k\right)-\sum\left(\lambda_{i}^{t}\right)^{2}+2k+1
\end{equation}
Where we use $\sum\lambda_{i}'=n+k,\sum\lambda_{i}=k$, and the last
$+1$ term appears because there is exactly one $i$ for which $\lambda_{i}^{t}=0,\lambda_{i}'^{t}\ge1$.
Combining (\ref{=T},\ref{=T2},\ref{=sq}) we obtain by a straightforward computation
\begin{equation}
\dim T=\frac{\left(n+k\right)^{2}-\sum_{i}\left(\lambda_{i}'^{t}\right)^{2}-k^{2}+\sum_{i}\left(\lambda_{i}^{t}\right)^{2}}{2}=\frac{\dim\Oc'-\dim\Oc}{2}
\end{equation}
as required.\\
\\
Let us now consider the second case, $\lambda_{1}'=\lambda_{1}$.
This case follows from a simpler version of the same argument, so
we only sketch the proof. Fix a decomposition $U=U_{1}\oplus U_{2}$
where $U_{1}$ corresponds to one maximal Jordan block of $K$ and
$U_{2}$ to all other blocks. Since the elements of $L$ cannot generate
longer chains than the maximal chains of $K$, any $A\in T$ satisfies
$AL\subset KU_{1}\oplus U_{2}$. Similarly to the argument we had
before, the map $L\to KU_{1}$ can be any map, while the map $L\to U_{2}$
should solve the smaller problem with spaces $U_{2},L$ and partitions
$\mu=\left(\lambda_{2},\lambda_{3},...\right),\mu'=\left(\lambda_{2}',\lambda_{3}',...\right)$.
If $T_{1}$ is the space of solutions to the smaller problem, we get
\begin{equation}
\dim T=\dim \Hom_{\C}\left(L,KU_{1}\right)+\dim T_{1}=n\left(\lambda_{1}-1\right)+\dim T_{1}
\end{equation}
Using the induction hypothesis for $T_{1}$, an analogous (but simpler)
computation shows that $\dim T=\frac{\dim\Oc'-\dim\Oc}{2}$, as required.
\proofend
}

\end{document}